\documentclass[preprint,aos]{imsart}
\RequirePackage{natbib}

\setattribute{journal}{name}{}

\RequirePackage[OT1]{fontenc}
\RequirePackage{amsthm,amsmath,natbib,graphicx,enumerate}
\RequirePackage[colorlinks,citecolor=blue,urlcolor=blue]{hyperref}
\RequirePackage{hypernat}
\usepackage{amssymb,tabularx,multicol,multirow,booktabs}
\usepackage{mhequ}
\usepackage[top=1.5in, bottom=1.5in, left=1in, right=1in]{geometry}
\usepackage[usenames,dvipsnames]{color}

\def \lhat {\hat{l}}
\def \lprime {l^{'}}

 \def\Par{{\mathbf{\mathcal{P}} }}
 \def \bmin {\beta_{\min}}
 \def \bmax {\beta_{\max}}
 \def \gmin {\gamma_{\min}}
 \def \gmax {\gamma_{\max}}

 \def \lmin {l_{\min}}
 \def \lmax {l_{\max}} 
 \def \gtilde {\tilde{g}}
  \def \ftilde {\tilde{f}}
 \def \lhat {\hat{l}}
 \def \lstar {l^*}

 \def \kprime {k^{'}}
 
\input binary_macro.tex
 \def\Par{{\mathbf{\mathcal{P}} }}

\begin{document}
\begin{frontmatter}
\title{Optimal Adaptive Inference in Random Design Binary Regression}
\runtitle{Binary regression}

\begin{aug}
	\author{\fnms{Rajarshi} \snm{Mukherjee}\ead[label=e1]{rmukherj@stanford.edu}}
\and
\author{\fnms{Subhabrata} \snm{Sen}\ead[label=e2]{ ssen90@stanford.edu}}
	
	\affiliation{Department of Statistics, Stanford University}
\end{aug}

\address{Sequoia Hall\\
	390 Serra Mall\\
	Stanford, CA 94305-4065\\
	\printead{e1}\\ \printead{e2}}


\begin{abstract}
	
We construct confidence sets for the regression function in nonparametric binary regression with an unknown design density-- a nuisance parameter in the problem. These confidence sets are adaptive in $L^2$ loss over a continuous class of Sobolev type spaces. Adaptation holds in the smoothness of the regression function, over the maximal parameter spaces where adaptation is possible, provided the design density is smooth enough. We identify two key regimes --- one where adaptation is possible, and one where some critical regions must be removed. We address related questions about goodness of fit testing and adaptive estimation of relevant infinite dimensional  parameters. 
\end{abstract}


\end{frontmatter}
In many epidemiological studies, a binary response variable $Y$
 is independently observed on a population of individuals along with multiple covariates $\bX$ to explain the variability in the response. In the context of epidemiological studies, the probability of
observing a specific outcome conditional on the covariates is often referred to as the propensity score. Estimating propensity score type functions from observed data is often of interest, and these estimates are subsequently used in multiple inferential procedures such as propensity score matching \citep{rosenbaum1983central}, inverse probability weighted inference \citep{robins1994estimation} etc. In the context of semiparametric inference for missing data type problems, a nice exposition to the importance of
understanding questions of similar flavor can be found in \cite{tsiatis2007semiparametric}. 

Historically, regression models with binary outcomes have been approached through both parametric \citep{mccullagh1989generalized} and nonparametric lenses \citep{signorini2004kernel,antoniadis2000nonparametric}. Although parametric regression has the natural advantage of being simpler in interpretation and implementation, it often lacks the desired complexity required to
capture varieties of dependence between covariates and outcomes. Nonparametric binary regression attempts to address this question, but it has its own share of shortcomings-- the two major concerns being dependence on a priori knowledge about the true underlying regression function class and ease of implementation. Motivated by these, in this paper we study inference (estimation, testing, and confidence sets) in binary regression problems under nonparametric models having random covariates with unknown design density, with primary focus on adaptation over function classes.

To fix ideas, suppose we observe data $(\mathbf{x}_i, y_i)_{i=1}^n$, where $\mathbf{x}_i \in [0,1]^d$ and $y_i \in \{0,1\}$. Consider the binary regression model 
\be\label{eqn:model}
\E(y|\bx)=\P\left(y=1|\bx\right)=f(\bx), \ y \in \{0,1\} , \ \bx \sim g. 
\ee
For the rest of the paper we assume $g$ to be absolutely continuous with respect to the Lebesgue measure on $[0,1]^d$. Owing to the binary nature of the outcomes, the model is completely parametrized by the tuple $(f,g)$ and admits a likelihood representation
\be\label{eqn:likelihood}
l(y,\bx|f,g)=f(\bx)^y(1-f(\bx))^{1-y}g(\bx).
\ee
We will be interested in making inferences about the regression function $f$ (treating $g$ as an unknown nuisance function), assuming $f$ and $g$ belong to Sobolev type spaces $B_{2,\infty}^{\beta}(M)$ and $B_{2,\infty}^{\gamma}(M')$ respectively--- see Section \ref{sec:wavelets} for a precise definition. 

It is worth noting that, whereas an adaptive inference framework for Gaussian and density settings is well studied (\cite{burnashev1979minimax}, \cite{ingster1994minimax}, \cite{spokoiny1996adaptive}, \cite{lepski1999minimax}, \cite{ingster2009minimax}, \cite{ingster2012nonparametric},  \cite{carpentier2015testing} for goodness of fit testing, \cite{lepski1992problems,lepskii1992asymptotically,lepskii1993asymptotically}, \cite{donoho1994ideal,donoho1995adapting}, \cite{lepski1997optimal}, \cite{johnstone2002function}, \cite{cai2012minimax} for adaptive estimation, and \cite{li1989honest}, \cite{low1997nonparametric}, \cite{baraud2004confidence}, \cite{cai2004adaptation}, \cite{robins2006adaptive}, \cite{gine2010confidence}, \cite{hoffmann2011adaptive}, \cite{bull2013adaptive}, \cite{szabo2015frequentist} for honest adaptive confidence sets), the  corresponding inferential questions in binary regression, with design density unknown, have received less attention.

 In many instances, results in estimation and hypothesis testing for a non-Gaussian setup might be derived from a related Gaussian setup by appealing to the theory of asymptotic equivalence of experiments. However, it is well known that such equivalence only takes effect above certain threshold of smoothness for the underlying functions of interest. Also, asymptotic equivalence of regression models with multidimensional covariates and random covariate density is a lesser studied subject. Therefore the question of adaptive  estimation for binary regression with multivariate random design cannot be addressed by simply invoking results from asymptotic equivalence. {Moreover, the theory of asymptotic equivalence of experiments does not throw any light on the construction of adaptive confidence balls --- one of the main questions of interest in this paper.
 	}

We also note that in contrast to the usual framework for random design gaussian regression problems, we consider a setup where the design density is unknown--- hence a nuisance parameter in the problem. Although \cite{carpentier2013honest} comments briefly on the case of nonparametric regression with uniformly random design density, these do not extend to the unknown design density case. Our setup, while being more realistic, makes our proofs technically more involved. 
 The basic heuristic for our analysis is that in case the unknown design density is smooth enough,  modulo certain modifications (to be made precise later),  the  ``effect of estimating" the unknown design density is negligible compared to the errors in making inference for the unknown regression function. 

In particular, the main results of this paper are summarized below.  
\begin{enumerate}[(a)]
	\item We produce estimators of underlying regression and design density which apart from jointly adapting over desired regimes of smoothness in an $L_2$ sense has the additional property of satisfying suitable boundedness (in both point-wise and Besov type norm sense) properties if the underlying functions are also similarly bounded (see Theorem \ref{thm:estimation_regression}). 

	\item We provide complete solution (lower and upper bounds) to the problem of asymptotic minimax goodness of fit testing 
	with both simple and composite null hypotheses (see Theorem \ref{thm:testing}) and unknown design density. 
	An analogous result (with sharp asymptotics) for \textit{simple null hypothesis} in Gaussian regression with multi-dimensional covariates with \textit{known} design density and regression function having at least $\frac{d}{4}$ derivatives was developed by \cite{ingster2009minimax}.
	
	
	\item We provide theory for adaptive confidence sets which complements those obtained in density \citep{bull2013adaptive} and sequence models \citep{robins2006adaptive,carpentier2013honest} (see Theorem \ref{theorem_adaptive_confidence_set}). A part of the adaptation theory for H\"{o}lder balls was sketched briefly in \cite{robins2008higher} using the theory of higher order influence functions, where honest adaptation was possible in parts of the parameter space. Our results are over Besov balls, where following ideas of \cite{bull2013adaptive}, we identify regions of the parameter space where adaptation is not possible without removing parts of the parameter space. We make this more precise in Section \ref{main_results}.
	
	\item All of our procedures are based on second order U-statistics constructed from projection kernels of suitable wavelet bases. We therefore extend the exponential inequality obtained in \cite{bull2013adaptive} to more general second order U-statistics based on wavelet projection kernels (See Lemma \ref{lemma_ustat_tail_use}). For the case of testing of composite alternatives \ref{test:comp_unknown}, this also adds to the chi-square type empirical wavelet coefficient procedure of \cite{carpentier2015regularity}.
\end{enumerate} 

\subsection*{Notation}
The results in this paper are mostly asymptotic in nature and thus requires some standard asymptotic  notations. If $a_n$ and $b_n$ are two sequences of real numbers then $a_n \gg b_n$ (and $a_n \ll b_n$) implies that ${a_n}/{b_n} \rightarrow \infty$ (respectively ${a_n}/{b_n} \rightarrow 0$) as $n \rightarrow \infty$. Similarly $a_n \gtrsim b_n$ (and $a_n \lesssim b_n$) implies that $\liminf{{a_n}/{b_n}} = C$ for some $C \in (0,\infty]$ (and $\limsup{{a_n}/{b_n}} =C$ for some $C \in [0,\infty)$). Alternatively, $a_n=o(b_n)$ will also imply $a_n \ll b_n$ and $a_n=O(b_n)$ will imply that $\limsup{{a_n}/{b_n}} =C$ for some $C \in [0,\infty)$). We comment briefly on the various constants appearing throughout the text and proofs. Given that our primary results concern convergence rates of various estimators, we will not emphasize the role of constants throughout and rely on fairly generic notation for such constants. In particular, for any fixed  tuple $v$ of real numbers, $C(v)$ will denote a constant depending on elements of $v$ only. Throughout the paper we shall use $\E_{P}$ and $\P_P$ to denote expectation and probability under the measure $P$, and $\I$ will stand for the indicator function. For any linear subspace $L\subseteq L_2[0,1]^d$, let $\Pi\left(h|L\right)$ denote the orthogonal projection of $h$ onto $L$ under the Lebesgue measure. Finally, for suitable functions $h: [0,1]^d \to \mathbb{R}$, we let $\|h\|_q:=(\int\limits_{[0,1]^d} |h(\bx)|^q d\bx)^{1/q}$ and $\|h\|_{\infty}:=\sup_{\bx \in [0,1]^d}|h(\bx)|$ denote the usual $L_{q}$ and $L_{\infty}$ semi-norm of $h$ respectively. 
\subsection*{Organization}
The rest of the paper is organized as follows. In Section 1 we describe the main results along with the definition of honest adaptive confidence sets. Section 2 discusses our choice of model and places it in the broader perspective of heteroscedastic nonparametric regression. We collect the technical details (definition of Besov type spaces along with discussion on compactly supported wavelet bases) and proofs of the main theorems in Section 3. In Section 4 we discuss the assumptions made in the paper and scope of future research. Finally we collect the proofs of certain technical lemmas in the appendix.
\section{Main Results}\label{main_results}
We outline our main results in this section. We work with certain smoothness classes for both the regression and design density with suitable additional assumptions on the boundedness. For conciseness of notation, we define,
\be
\Par(\beta, \gamma,M,M',B_L,B_U)= \left\{\begin{array}{c} (f,g) : f \in B_{2,\infty}^{\beta}(M), g \in B_{2,\infty}^{\gamma}(M'), 0<f<1,\\ 0<B_L \leq g \leq B_U, \int g(\bx)d\mathbf{x} =1\end{array}\right\}. \label{eqn:parameter_space}
\ee
 Above and throughout the paper, by the pair $(f,g)$ we shall refer to the probability measure $P$ generated according to \ref{eqn:likelihood} by the regression function $f$ and marginal density $g$ respectively. Therefore, by an abuse of notation, we will refer to the elements of $\Par$ interchangeably as either the pair $(f,g)$ or the corresponding probability measure $P$. We will always assume that the radius and boundedness parameters  $(M,M',B_L,B_U)$  are known to us. There are indeed some subtleties involved in inference without the knowledge of these parameters. These issues can be dealt with using our arguments adapted to Theorem 4 of \cite{bull2013adaptive}. The lower and upper bound requirements on the design density can also be relaxed to a certain extent at the cost of more involved proofs. However, for focused discussion, these will not  be addressed in this paper.  For notational brevity, we will henceforth denote $\Par(\beta, \gamma,M,M',B_L,B_U)$ simply as $\Par(\beta, \gamma)$. 

\subsection{Adaptive Estimation of parameters}
Our first result establishes the existence of certain rate optimal estimators for the regression and design density in our setup. We further establish that these estimators satisfy certain additional boundedness properties almost surely, which is invaluable for subsequent inference in this setup. 
 \begin{theorem}
\label{thm:estimation_regression}
\noindent
\begin{enumerate}
\item \label{thm:estimation_part1}
 Let $0<\gamma_{\min}<\gamma_{\max}$ be given. There exists a sequence of estimators $\hat{g} = \hat{g}(\bx_1, \cdots, \bx_n,B_U,\gmin,$ $ \gamma_{\max})$ of the design density $g$ and constant $C=C(M',\gmin, \gamma_{\max},B_U)$ such that for each $\gamma \in [\gamma_{\min},\gamma_{\max}]$ and $\beta>0$, 
\begin{align}
\sup_{P \in \Par(\beta, \gamma)} \E_{P}\left[\|\hat{g} - g\|_2^2\right] \leq C n^{-\frac{2\gamma}{2 \gamma + d}}, \nonumber \\
\liminf_{n \to \infty} \inf_{P \in \Par(\beta, \gamma)}\P_P[ \hat{g} \in B_{2,\infty}^{\gamma}(C)] = 1, \nonumber
\end{align}
and there exists constants $0< B_L' \leq B_U'$ (depending on $B_L,B_U$) such that $B_L' \leq \hat{g} \leq B_U'$ almost surely. Further, there exists a universal constant $c>0$ such that 
\begin{align}
\inf_{\hat{g}} \sup_{P\in \Par(\beta, \gamma)} \E_{P}[\| \hat{g} - g \|_2^2] \geq c n^{- \frac{2\gamma}{ 2\gamma +d }}. \nonumber
\end{align}

\item \label{thm:estimation_part2}
Let $0<\beta_{\min}\leq \beta_{\max}$, $\gamma_{\min}<\gamma_{\max}$ be given. If $\gamma_{\min}> \beta_{\max}$, there exists a sequence of estimators $\fhat= \hat{f}_n(\bx_1, y_1, \cdots, \bx_n, y_n, M,M', B_U,B_L,\bmin,\bmax,\gmax)$ and constant $C=C(M,M',B_U,B_L,\bmin,\bmax,\gmax)$ such that for every $\beta \in [\betamin,\betamax]$ and $\gamma \in  [\gamma_{\min},\gamma_{\max}]$
\begin{align*}
&\sup\limits_{P \in \Par(\beta,\gamma)} \E_P\left[\|\fhat-f\|_2^2\right] \leq Cn^{-\frac{2\beta}{2\beta+d}}, \\
&\liminf_{n \to \infty} \inf_{P \in \Par(\beta, \gamma)} \P_P[\fhat \in \besov^{\beta}(C)] =1,
\end{align*}
and there exist constants $C_L \leq C_U$ such that $C_L \leq \hat{f} \leq C_U$ almost surely. 
Further, there exists a constant $c>0$, independent of $n$, such that 
	\begin{align*}
		\inf_{\fhat} \sup_{P \in \Par(\beta,\gamma)} \E_{P} [ \| \hat{f}- f \|_2^2 ] \geq c n^{- \frac{2 \beta}{2\beta + d}}.  \nonumber
	\end{align*}
\end{enumerate}
\end{theorem}

The proof of Theorem \ref{thm:estimation_regression} is outlined in Section \ref{section:estimation_proof}. 

\begin{remark}
The result of Theorem \ref{thm:estimation_regression} part \ref{thm:estimation_part1} is similar to that in \cite{bull2013adaptive}. {It is worth noting that results of the kind stating that $\ghat \in B_{2,\infty}^{\gamma}(M^{*})$ with high probability uniformly over $\Par(\beta,\gamma)$ for a suitably large constant $M^{*}$ is not too hard to show. However, our proof shows that a suitably bounded estimator $\ghat$, which adapts over smoothness and satisfies $\ghat \in B_{2,\infty}^{\gamma}(M^{*})$ with probability larger than $1-\frac{1}{n^{\theta}}$ uniformly over $\Par(\beta,\gamma)$, for any $\theta>0$ and correspondingly large enough $M^{*}$ . Additionally, the results of Theorem \ref{thm:estimation_regression} part \ref{thm:estimation_part2} are relatively less common in an unknown design density setting. Indeed, adaptive estimation of regression function with random design over Besov type smoothness classes has been obtained by model selection type techniques by \cite{baraud2002model} for the case of Gaussian errors. Our results in contrast, as remarked in Section \ref{section_discussion},  hold for any regression model with bounded outcomes and compactly supported covariates having suitable marginal design density.}  
\end{remark}

\begin{remark}
The dependence of our constants on $\gmax$ stems from deciding the regularity of the wavelet basis used. Once we fix a wavelet basis with regularity $S>\gmax$, the dependence of our constants on $\gmax$ can be reduced to dependence on $S$.
\end{remark}
\subsection{Construction of Confidence Sets}
To tackle the question of adaptive confidence sets in our setup, we need to first analyze the goodness of fit problem in this setup. The next theorem characterizes the minimax testing rate for our problem. The proof is deferred to Section \ref{sec:proof_testing}. To this end, we introduce the parameter spaces
\be
\Par_0(\beta,\gamma)&= \{ (f,g): f \equiv1/2 , g \in B_{2,\infty}^{\gamma}(M'), B_L < g < B_U, \int g(\bx) d\bx =1\}, \\
 \Par(\beta,\gamma, \rho_n^2) &= \{ (f,g): f \in \besovb(M),\Big\|f - \frac{1}{2} \Big\|_2^2 >\rho_n^2, g \in B_{2,\infty}^{\gamma}(M'), B_L < g < B_U, \int g(\bx) d\bx =1\}. 
 \ee
 Further, for $\beta_1 > \beta_2$, we define, 
 \begin{align}
 \Par(\beta_1, \beta_2, \gamma, \rho_n^2) =  
\left\{ \begin{array}{c}  (f,g): f \in B_{2,\infty}^{\beta_2}(M), \Big\|f - B_{2,\infty}^{\beta_1}(M) \Big\|_2^2 >\rho_n^2, g \in B_{2,\infty}^{\gamma}(M'),  \\ 
 B_L < g < B_U, \int g(\bx) d\bx =1.  \end{array} \right\}.  \nonumber 
\end{align}
Finally we recall $\Par(\beta,\gamma)$ defined in \eqref{eqn:parameter_space}.
\begin{theorem}
\label{thm:testing}
\noindent
\begin{enumerate}
\item
\label{thm:testing_simple}
 Consider the testing problem
\be
	H_0 : P \in \Par_0(\beta,\gamma) \,\,\,\,vs.\,\,\,\, H_1: P \in \Par(\beta,\gamma, \rho_n^2), 
\ee
for $\gamma>\beta$.
We have, 
	\begin{itemize}
		\item For any $0 < \alpha <1 $, there exists $D>0$ sufficiently large (depending on $\alpha, M, M'$)  and  a test $\phi$ such that for $\rho_n^2 = D n ^{- \frac{4\beta}{ 4\beta +d }} $ 
		\be
		\limsup_{n \to \infty}	\Big(\sup_{P \in \Par_0(\beta,\gamma)} \P_{P}[\phi =1 ] +\sup_{P \in \Par(\beta,\gamma, \rho_n^2)} \P_{P} [ \phi =0 ] \Big) \leq \alpha \label{eq:error} . 
		\ee

		\item For any test $\phi$ which satisfies \eqref{eq:error} introduced above, the corresponding sequence 
		$\rho_n^2$ satisfies 
		\be
			 \liminf_{n\to \infty} \rho_n^2 \gtrsim n^{- \frac{4\beta}{4 \beta +d } } . 
		\ee
\end{itemize}

\item
\label{thm:testing_composite}
 Consider the testing problem 
 \be
 	H_0 : P \in \Par(\beta_1, \gamma) \,\,\, vs. \,\,\, H_1: P \in \Par(\beta_1,\beta_2, \gamma, \rho_n^2), \label{test:comp_unknown}
 \ee
 for {$  \beta_2  < \beta_1 $} and {$\gamma > 2 \beta_2$}. Then 
 	\begin{itemize}
 		\item For any $0 < \alpha <1 $, there exists $D>0$ sufficiently large (depending on $\alpha, M, M'$)  and  a test $\phi$ such that for $\rho_n^2 = D n ^{- \frac{4\beta_2}{ 4\beta_2 +d }} $ 
 		\be
 		\limsup_{n\to \infty} \Big[	\sup_{P \in \Par(\beta_1,\gamma)}\P_{P}[\phi =1 ] + \sup_{P \in \Par(\beta_1,\beta_2,\gamma, \rho_n^2)} \P_{P} [ \phi =0 ]  \Big]\leq \alpha \label{eq:error_unknown} .
 		\ee
 		
 		\item  \label{thm:comp_unknown2} 
 		For any test $\phi$ which satisfies \eqref{eq:error_unknown} introduced above, the corresponding sequence 
 		$\rho_n^2$ satisfies 
 		\be
 			\liminf_{n \to \infty} \rho_n^2 \gtrsim n^{- \frac{4\beta_2}{4 \beta_2 +d } } .
 		\ee
 	\end{itemize}  
\end{enumerate}
\end{theorem}
A few remarks are in order about the results above. First, it is interesting to note whether the complexity of the null hypothesis affects the minimal rate of separation between the null and the alternative necessary to carry out the test. Our result answers this question in the negative. As mentioned earlier, although the results appear to be of similar flavor to those in \cite{bull2013adaptive,carpentier2015regularity}, the rigorous derivations require careful understanding and modifications to accommodate for the effect of estimating an unknown density. 
A possible approach to the testing problem \eqref{test:comp_unknown} can be the method of \cite{carpentier2015regularity} without further modification. However, such an approach results in unbiased estimation of $\|\Pi(fg|L)\|_2^2$ for appropriate subspaces $L \subset L_2[0,1]^d$ instead of $\|\Pi(f|L)\|_2^2$ required for understanding the minimum separation $\Big\|f - B_{2,\infty}^{\beta_1}(M) \Big\|_2^2$. {Instead, our proof shows that under the alternative, the quantity $\|\Pi(f\frac{g}{\ghat}|L)\|_2^2$ is also large enough for suitable subspaces $L$. This quantity is easier to estimate modulo the availability of a nice estimator $\ghat$--- which is in turn guaranteed by Theorem \ref{thm:estimation_regression}.} However, this also necessitates modifying the testing procedure of \cite{carpentier2015regularity} suitably to incorporate the effect of estimating $g$. We make this more clear in the proof of Theorem \ref{thm:testing_composite}.

Next, we outline the construction of honest adaptive confidence sets in our setup.  We briefly introduce the relevant notions for convenience. A confidence set $C_n = C(\bx_1, y_1, \cdots, \bx_n, y_n)$ is a random measurable subset of $L^2$. We define the $L^2$ radius of a set $C$ as 
\begin{align*}
|C|= \inf\{ \tau: C\subset \{ \psi : \| \psi - g \|_2 \leq \tau\} \textrm{ for some } g\}. 
\end{align*}
We seek to determine the maximal parameter spaces $\Par_n$ so that adaptive confidence sets exist. We define a confidence set $C_n= C_n(\bx_1, y_1 ,\cdots, \bx_n, y_n)$ to be \textit{honest} over a sequence of models $\Par_n$  if 
\be
\inf_{ P \in \Par_n } \E_P[ f \in  C_n] \geq 1- \alpha - r_n, \label{eqn:honesty}
\ee
where $r_n \to 0$ as $n \to \infty$ and $\alpha<1$ is a fixed level of confidence. Further, we call a confidence set $C_n$ adaptive over a sequence of models $\Par_n$ if there exists a constant $C$ depending on the known parameters of the model space $\Par_n$ such that 
\be
\sup_{P \in \Par_n \cap \Par(\beta, \gamma)} \P_P\Big[ |C_n|^2 \geq C n^{-\frac{2\beta}{2\beta +d }}\Big] \leq \alpha', \label{eqn:adapt}
\ee
where $0<\alpha'<1$ is a fixed constant.

Now we define the parameter spaces over which we will produce honest adaptive confidence sets in $L_2$. For given interval of smoothness of regression function $[\beta_{\min},\beta_{\max}]$ such that $\beta_{\max}>2\beta_{
\min}$, we define a grid following ideas from \cite{bull2013adaptive}. With $N=\lceil \log_2\left(\frac{\beta_{\max}}{\beta_{\min}}\right) 
\rceil$ define $\beta_j=2^{j-1}\beta_{\min}, \ j=1,\ldots,N$. With this notation, we define
\be
\F_n(M^*)&=\besov^{\beta_N}(M)\cup \Big(\bigcup_{j=1}^{N-1}\mbesov^{\beta_j}(M,M^*\rho_n(\beta_j))\Big), \nonumber\\ 
\Par_n(M^*,M',\gamma) &= \Big\{ (f,g): f \in \F_n(M^*), 0< f <1, g \in \besov^{\gamma}(M'), B_L \leq g \leq  B_U, \int g(\bx) d\bx=1\Big\}, \nonumber 
\ee
where  ${{\mbesov^{\beta_j}(M,M^*\rho_n(\beta_j))=\left\{ h \in B_{2,\infty}^{\beta_j}(M) : \| h - B_{2,\infty}^{\beta_{j+1}}(M) \|_2 \geq M^*\rho_n(\beta_j)\right\}}}$, $\rho_n(\beta)=n^{-\frac{2\beta/d}{4\beta/d+1}}$, and the choice of $M^*$ is solely guided by $M,M',\beta_{\min},\beta_{\max},B_L,B_U$ and can be read off from the proof of the next theorem.

\begin{theorem}\label{theorem_adaptive_confidence_set}
Let $0<\beta_{\min}\leq \beta_{\max}$, $ 2 \beta_{\max}< \gamma_{\min} \leq \gamma_{\max}$ be given. Then there exists a confidence set $C_n$ depending only on the tuple $(M, M',\beta_{\min},\beta_{\max}, \gmin, \gamma_{\max},B_L,B_U,\alpha,\alpha^{'})$ which is honest and adaptive in the sense of \eqref{eqn:honesty} and \eqref{eqn:adapt} over $\bigcup_{\gamma \in [\gmin, \gmax]}\Par_n(M^*,M',\gamma)$, whenever $M^*$ is large enough.
 \end{theorem}

Theorem \ref{theorem_adaptive_confidence_set} is proved in Section \ref{sec:confidence_proof}.
It is of interest to determine whether the models $\cup_{\g}\Par_n$, are in some sense, the maximal spaces over which adaptation is possible. We note that the testing lower bounds established in Theorem \ref{thm:testing} part \ref{thm:testing_composite} above imply that $\cup_{\g}\Par_n$ is indeed the largest parameter space, up to multiplicative constants of $\rho_n(\beta_j), j=1,\ldots,N$, over which adaptation is possible. {Moreover, results of the flavor of \cite[Theorem 3]{bull2013adaptive} can be recovered from the proof of Theorem \ref{theorem_adaptive_confidence_set}.}

\section{Choice of Binary Regression Model}\label{section:choice_of_binary_regression} In this section we comment on our choice of binary regression model.

Regarding the generality of our model choice, there are two main points that need addressing. The first concerns the framing of model \eqref{eqn:model} without going through a link function--- as is the general custom for generalized linear models. Indeed for a link function formulation as
\be\label{eqn:model_glm}
\E(y|\bx)=\theta(h(x)), \ y \in \{0,1\} , \ \bx \sim g,
\ee
for $\theta$ a distribution function of a symmetric random variable (probit, logistic etc.), our results still go through provided $\theta$ satisfies some regularity conditions. In general for a smooth function $\theta$, the function $\theta(h(x))$ shares the smoothness index of $h$ and identifying $f:=\theta\circ h$ lands us back in model \eqref{eqn:model}. To keep things simple, we work with model \eqref{eqn:model} throughout. 

The second point to note is that we have not considered the fixed design case in our set up. The fixed design problem can be addressed similarly with more straightforward generalization of ideas from \cite{robins2006adaptive}, \cite{bull2013adaptive}, and \cite{carpentier2015regularity} due to lack of extra nuisance parameter $g$. We omit this for the sake of brevity.

The other point worth discussing concerns the generalizability of the binary regression model to more general nonparametric regression models. In this context note that, additive Gaussian noise is the simplest example of  a situation where the regression function can be parametrized separately from other components of the model such as conditional variance given the covariates. This facilitates the development of a satisfying adaptation theory, even over large classes of unknown design densities. In non-gaussian settings, given a likelihood specification for the conditional distribution of outcomes given covariates, one can  attempt to produce a similar theory for adaptive inference. 

As a step towards a general theory of adaptive inference in nonparametric regression, we consider the case of binary outcomes. Binary regression automatically belongs to a heteroscedastic variance regime--- a more challenging scenario in general (\cite{efromovich1996nonparametric}, \cite{efromovich1996sharp}, \cite{antoniadis2001wavelet}, \cite{antoniadis2001waveletq},\\ \cite{galtchouk2008adaptive}, \cite{galtchouk2009sharp}). However, it is different from standard heteroscedastic additive Gaussian noise regression problems in that the mean regression function is intimately tied to the conditional variance and shares the same smoothness. In this case, the simplicity of the conditional distribution of the outcome given regressors allows us to answer the question of adaptive inference to some degree of generality.

{Finally we remark that most of our results actually hold not only for the case of binary regression, but also  for any regression model with bounded outcomes and compactly supported covariates having suitable marginal design density. It is for the proof of matching lower bounds to show that our results are asymptotically rate optimal, that we need the binary regression model. }

 \section{Discussion}\label{section_discussion}
 Although we have tried to describe adaptive confidence sets for binary regression to some degree of generality, it is instructive to discuss some of the assumptions made in the process.  Throughout our paper, we assume a lower bound of smoothness on the marginal density $g$ of $\bx$. Although our assumption is not sharp, we believe that such an assumption on the marginal density of $\bx$ is necessary to a certain extent. This ensures that we learn about $g$ at a rate fast enough so that it does not reflect too adversely on the inference for $f$. One can also wonder if the requirement $\gamma_{\min}>2\beta_{\max}$ in Theorem \ref{theorem_adaptive_confidence_set} can be relaxed.  Using results from \cite{robins2008higher} it is possible to further reduce our lower bound on the smoothness of $g$ in the context of adaptive confidence sets over smoothness of $f$ satisfying $\beta_{\min}<\beta_{\max}<2\beta_{\min}$. It remains an interesting and challenging question to understand the sharp lower bound on the smoothness for $g$ under which one no longer derive results similar to those obtained here in the other regime i.e. $2\beta_{\min}<\beta_{\max}$. In a future project we plan to investigate this issue with special focus on using higher order influence functions \citep{robins2008higher, robins2015higher}. Indeed, even in the case of non-adaptive inference, \cite{robins2008higher, robins2015higher} require certain smoothness lower bounds on the unknown design density. {A related point of view for constructing honest adaptive confidence sets is often in the context of ``self similar" functions-- a case where construction of fully adaptive honest confidence sets are possible without further removing parts of the self similar function spaces \citep{picard2000adaptive,gine2010confidence,kerkyacharian2012concentration,bull2012honest,nickl2014sharp,ray2014bernstein,szabo2015frequentist}.} Although we do not pursue this in our paper, it is possible to use ideas from our paper to answer similar questions.

\section{Technical Details}
\subsection{Wavelets and Besov Spaces}
\label{sec:wavelets}
In this section, we collect some facts about wavelets and Besov spaces. We also introduce some notation that we use later. 
For $d>1$, consider expansions of functions $h \in L_2\left([0,1]^d\right)$ on an orthonormal basis of compactly supported bounded wavelets of the form 
\be
h(\bx)&=\sum_{k \in \mathbb{Z}^d}\langle h, \bpsi_{0,k}^0\rangle \bpsi_{0,k}^0(\bx)+ \sum_{l=0}^{\infty}\sum_{k \in \mathbb{Z}^d}\sum\limits_{v \in \{0,1\}^d-\{0\}^d}\langle h, \bpsi_{l,k}^v\rangle \bpsi_{l,k}^v(\bx) , 
\ee
where the base functions $\bpsi_{l,k}^v$ are orthogonal for different indices $(l,k,v)$ and are scaled and translated versions of the $2^d$ $S$-regular base functions $\bpsi_{0,0}^v$ with $S>\beta$, i.e., $\bpsilk^v(x)=2^{ld/2}\bpsi_{0,0}^v(2^l \bx-k)=\prod_{j=1}^{d}2^{\frac{l}{2}}\psi_{0,0}^{v_j}\left(2^lx_j-k_j\right)$ for $k=(K_1,\ldots,k_d) \in \ZZ^d$ and $v=(v_1,\ldots,v_d)\in \{0,1\}^d$ with $\psi_{0,0}^0=\phi$ and $\psi_{0,0}^1=\psi$ being the scaling function and mother wavelet of regularity $S$ respectively as defined in one dimensional case. As our choices of wavelets, we will throughout use compactly supported scaling and wavelet functions of Cohen-Daubechies-Vial type with $S$ first
null moments\citep{cohen1993wavelets}. In view of the compact support of the wavelets, for each resolution level $l$ and index $v$, only $O(2^{ld})$ base elements $\psi_{l,k}^v$ are non-zero on $[0,1]$; let us denote the corresponding set of indices $k$ by $\Z_l$ obtaining the representation,
\be
h(\bx)&=\sum_{k \in \Z_{J_0}}\langle h, \bpsi_{J_0,k}^0\rangle \bpsi_{J_0,k}^0(\bx)+ \sum_{l=J_0}^{\infty}\sum_{k \in \Z_l}\sum\limits_{v \in \{0,1\}^d-\{0\}^d}\langle h, \bpsi_{l,k}^v\rangle \bpsi_{l,k}^v(\bx) , \label{eqn:wavelet_expansion_aad_compact_j0}
\ee
where $J_0=J_0(S)\geq 1$ is such that $2^{J_0}\geq S$ \citep{cohen1993wavelets, gine2015mathematical}. Thereafter, letting for any $h \in L_2[0,1]^d$, $\|\langle h,\bpsi_{\lprime,\cdot}\rangle\|_2$ be the vector $L_2$ norm of the vector \\$\left(\langle h,\bpsi^v_{\lprime,\kprime}\rangle: \kprime \in \mathcal{Z}_{\lprime},v \in \left\{0,1\right\}^d\right)$, define
\be
\besovb (M):=\Big\{h \in L_2\left([0,1]^d\right):  \|h\|_{\beta,2}:=2^{J_0\beta}\|\langle h,\bpsi^{0}_{J_0,\cdot}\rangle\|_2+\sup_{l \geq J_0}2^{l\beta}\Big(\sum_{k \in \mathbb{Z}^d}\sum_{v \in \{0,1\}^d-\{0\}^d}\langle h, \bpsi_{l,k}^v\rangle^2\Big)^{\frac{1}{2}}\leq M \Big\}. \\ \label{eqn:besov_multidim}
\ee
We will be working with projections onto subspaces defined by truncating expansions as above at certain resolution levels. For example letting
\be
V_j:=\mathrm{span}\left\{\bpsilk^v, J_0 \leq l\leq j, k \in \Z_l, v\in \{0,1\}^d\right\}, j \geq J_0 \label{eqn:defn_vj}
\ee
one immediately has the following orthogonal projection kernel onto $V_j$ as 
\be
K_{V_{j}}\left(\bx_1,\bx_2\right)=\sum_{k \in \Z_{J_0}}\bpsi_{J_0,k}^0(\bx_1)\bpsi_{J_0,k}^0(\bx_2)+\sum_{l=J_0}^{j}\sum_{k \in \Z_l}\sum_{v \in \{0,1\}^d-\{0\}}\bpsi_{l,k}^v(\bx_1)\bpsi_{l,k}^v(\bx_2). \label{eqn:projection_vj}
\ee
Owing to the MRA property of the wavelet basis, it is easy to see that $K_{V_j}$ has the equivalent representation as 
\be
K_{V_{j}}\left(\bx_1,\bx_2\right)=\sum_{k \in \Z_j}\sum_{v \in \{0,1\}^d}\psi_{jk}^v\left(x_1\right)\psi_{jk}^v\left(x_2\right). \label{eqn:projection_vj_alt}
\ee
We will also consider,
\be
W_j:=\mathrm{span}\left\{\bpsi_{j,k}^v, k \in \Z_j, v\in \{0,1\}^d-\{0\}^d\right\}, j \geq J_0 \label{eqn:defn_wj}
\ee
and the corresponding orthogonal projection kernel onto $W_j$ as 
\be
K_{W_{j}}\left(\bx_1,\bx_2\right)=\sum_{k \in \Z_j}\sum_{v \in \{0,1\}^d-\{0\}^d}\bpsi_{j,k}^v(\bx_1)\bpsi_{j,k}^v(\bx_2). \label{eqn:projection_wj}
\ee 

\subsection{Proof of Theorem \ref{thm:testing}}
\label{sec:proof_testing}
We will describe the proof of Theorem \ref{thm:testing} in this section. To this end, we will crucially utilize Lemma \ref{lemma_ustat_tail_use}. The proof will be deferred to the Appendix B. 
The U-statistics appearing in this paper are mostly based on projection kernels sandwiched between arbitrary bounded functions. This necessitates generalizing the U-statistics bounds obtained in \cite{bull2013adaptive}. In particular, we are interested in tail bounds of U-statistics based on kernel $R(\bW_1,\bW_2)=L\left(\bW_1\right) K_{V_j}\left(\bX_1,\bX_2\right)L\left(\bW_2\right)$ and $R(\bW_1,\bW_2)=L\left(\bW_1\right) K_{W_j}\left(\bX_1,\bX_2\right)L\left(\bW_2\right)$ where $\bW=\left(Y,\bX\right)$ and $Y \in \mathbb{R}, \bX \in [0,1]^d$. Assume that $|L(\bW)|\leq B$ (which corresponds to our situation). 

\begin{lemma}\label{lemma_ustat_tail_use}
	There exists constant $C:=\constant>0$ such that
	\begin{align}
	\mathbb{P}\Big(\Big|\frac{1}{n(n-1)}\sum_{i_1 \neq i_2}R\left(\bW_{i_1},\bW_{i_2}\right)-\E\left(R\left(\bW_1,\bW_2\right)\right) \Big|\geq t\Big) & \leq e^{-C nt^2}+e^{-\frac{C t^2}{a_1^2}}+e^{-\frac{C t}{a_2}}+e^{-\frac{C \sqrt{t}}{\sqrt{a_3}}} \nonumber
	\end{align}
	where $a_1= \frac{1}{n-1}2^{\frac{jd}{2}}$, $a_2=\frac{1}{n-1}\Big(\sqrt{\frac{2^{jd}}{n}}+1\Big)$, $a_3=\frac{1}{n-1}\Big(\sqrt{\frac{2^{jd}}{n}}+\frac{2^{jd}}{n}\Big)$, $R(\bW_1,\bW_2)=L(\bW_1) K_{V_j}\left(\bX_1,\bX_2\right)L(\bW_2)$ or $R(\bW_1,\bW_2)=L(\bW_1) K_{W_j}\left(\bX_1,\bX_2\right)L(\bW_2)$ with $K_{V_j}$ and $K_{W_j}$ constructed using compactly supported wavelet bases of regularity $S$, $\bW=\left(Y,\bX\right)$, $|L(\bW)|\leq B$ almost surely $\bW$, and $\bX \in [0,1]^d$ has density $g$ such that $g(\bx)\leq B_U$ for all $\bx \in [0,1]^d$.
\end{lemma}

\subsubsection{Proof of Part \ref{thm:testing_simple}}
We will first introduce a test with the desired properties. 
 We use the statistic
	\begin{align*}
		T= \frac{1}{n(n-1)} \sum_{ 1\leq i \neq j \leq n } { (y_i - 1/2)} K_{V_{j_0}} (\mathbf{x}_i , \mathbf{x}_j ) {(y_j - 1/2) } . 
	\end{align*}
	Here, we choose $j_0 = \lceil \frac{2}{4\beta +d } \log_2 n \rceil$. 
	We reject this test when $|T| > C \frac{2^{j_0 d/2}}{n}$, for some constant $C$ to be chosen appropriately. We first control the Type I error for this test. We have, under $P \in \Par_0(\beta,\gamma)$, $\E_{ P}[T]=0$. 
	Applying Lemma \ref{lemma_ustat_tail_use}, we obtain $\P_{P}[ |T| > C \frac{2^{j_0 d/2}}{n}] \leq e^{-C}$. 
	Thus the Type I error may be controlled at the desired level $\alpha$ by choosing the cut-off $C$ sufficiently large. To control the Type II error, we fix $P \in \Par(\beta,\gamma, \rho_n^2)$. In this case, we have, 
	\begin{align}
		\E_{P}[ T ] &= \Big\| \Pi_{V_{j_0}}\Big( (f- 1/2) {g}\Big) \Big\|_2^2 
		= \Big\| (f - 1/2) {g}\Big \|_2^2 - \Big\| \Pi_{V_{j_0}^\perp} \Big((f-1/2) {g} \Big)\Big\|_2^2. \nonumber\\
		&\geq B_L^2 \rho_n^2 -  \frac{2C(M,M')}{\sqrt{1- 2^{-2\beta}}} 2^{- 2j_0 \beta},\nonumber   
	\end{align}
	where the last line follows since $\gamma>\beta$, using arguments similar to the proof of Lemma \ref{lemma:density_truncation}.

	Thus we have, 
	\begin{align*}
		\P_P\Big[ |T| > C \frac{2^{j_0 d/2}}{n} \Big] &= 1 - \P_P\Big[ |T| \leq C \frac{2^{j_0 d/2}}{n} \Big], \nonumber \\
		\P_P\Big[ |T| \leq C \frac{2^{j_0 d/2}}{n} \Big] 
		&\leq \P_P \Big[ | T - E_{P}[T] | \geq E_{P}[T] - C \frac{2^{j_0 d/2}}{n}   \Big] , \nonumber \\
		&\leq \P_P \Big[ |T- E_{P}[T] | \geq B_L^2\rho_n^2 - C' n^{-\frac{4\beta}{4\beta + d} }\Big],
	\end{align*}
	where $C'$ depends on $B_L,M,M'$.
	The proof is completed by an application of Lemma \ref{lemma_ustat_tail_use}, upon setting $\rho_n^2 =  D n^{- \frac{4\beta}{4\beta +d }}$ for some constant $D$ sufficiently large. 
	
Next, we establish a matching (up to constants) lower bound on the testing rate for this problem. Assume that $\rho_n^2 \ll n^{-\frac{4\beta}{4\beta+d}}$. The proof of the lower bound is then based on Theorem 2.1 of 
\cite{robins2009semiparametric}.  In particular, let $H:[0,1]^d\rightarrow \mathbb{R}$ be a $C^{\infty}$ function supported on $\left[0,\frac{1}{2}\right]^d$ such that $\int H(\bx)d\bx=0$ and $\int H^2(\bx)d\bx=1$ and let $k=\lceil c_0 n^{\frac{2d}{4\beta+d}}\rceil$ for some $c_0 >0$. Now suppose that $\Omega_1,\ldots,\Omega_k$ be the translates of the cube $k^{-\frac{1}{d}}\left[0,\frac{1}{2}\right]^d$ that are disjoint and contained in $[0,1]^d$. Letting $\bx_1,\ldots,\bx_k$ denote the bottom left corners of these cubes, we set for $\lambda=(\lambda_1,\ldots,\lambda_k)\in \{-1,+1\}^k$,
	$$f_{\lambda}=\frac{1}{2}+\Big(\frac{1}{k}\Big)^{\frac{\beta}{d}}\sum\limits_{j=1}^k\lambda_j H\Big((\bx-\bx_j)k^{\frac{1}{d}}\Big).$$
	The construction ensures that $f_{\lambda}\in B_{2,\infty}^{\beta}(M)$ ($H$ can be chosen to guarantee desired $M$) for every $\lambda=(\lambda_1,\ldots,\lambda_k)\in \{-1,+1\}^k$ and $\big\|f_{\lambda}-\frac{1}{2}\big\|_2^2 = \left(\frac{1}{k}\right)^{\frac{2\beta}{d}}$. Therefore, by the choice of $k$, each $f_{\lambda}$ corresponds to a measure in the alternative hypothesis. Choose $\pi$ to be the uniform prior on $\{-1,+1\}^k$. We use the notation of Theorem 2.1 of \cite{robins2009semiparametric}, 
	let us partition the sample space $\chi=\{0,1\}\times [0,1]^d$ into $\chi_j=\{0,1\}\times \Omega_j, \ j=1,\dots,k$ and the remaining set. Letting $P_{\lambda}$ and $Q_{\lambda}$ be the probability measure on $\{0,1\}\times [0,1]^d$ corresponding to  likelihood \eqref{eqn:likelihood} for $f=f_{\lambda}$ and $f\equiv\frac{1}{2}$ respectively, its obvious that $P_{\lambda}(\chi_j)=Q_{\lambda}(\chi_j)=p_j \ (\text{say})$, since $\int H(\bx)d\bx=0$. Also, $p_j \in \frac{1}{k}[\underline{B},\overline{B}]$ for fixed constants $\underline{B},\overline{B}$. Moreover, $\delta=\max_j\sup_{\lambda}\int_{\chi_j}\frac{(q-p)^2}{p_{\lambda}}\frac{d\mu}{p_j}$ since $p=\int p_{\lambda}d\pi(\lambda)=\int f_{\lambda}^y(1-f_{\lambda})^{1-y}d\pi(\lambda)=\frac{1}{2}=q$. Finally, $p_{\lambda}-q=p_{\lambda}-p=\left(f_{\lambda}-\frac{1}{2}\right)^y\left(\frac{1}{2}-f_{\lambda}\right)^{1-y}$ implies that $a=b=\max_j\sup_{\lambda} \int_{\Omega_j}\frac{\left(f_{\lambda}-\frac{1}{2}\right)^2}{p_j}\in k^{-\frac{2\beta}{d}}[\underline{B},\overline{B}]$. Therefore, by Theorem 2.1 of \cite{robins2009semiparametric} if $\rho(P_1,P_2)$ denotes the Hellinger affinity between two probability measures $P_1,P_2$ defined on the same probability space 
	$$\rho\Big(\int P_{\lambda}d(\pi(\lambda)),\int Q_{\lambda}d(\pi(\lambda))\Big)\geq 1-C\frac{n^2}{k}k^{-\frac{4\beta}{d}},$$
	which can be made arbitrarily close to one for large enough $c_0$. This proves the theorem since if the Hellinger affinity is bounded away from $1$, then there does any consistent sequence of tests distinguishing between the null hypothesis and the easier alternative corresponding to the $f_{\lambda}$'s constructed above \citep{tsybakov2008introduction}. 

\subsubsection{Proof of Part \ref{thm:testing_composite}} 
We will  construct a test with the desired properties below. The proof of the testing lower bound follows from the argument outlined for the previous part of the Theorem. Our proof is similar in spirit to that of \cite{carpentier2015regularity}, though the details are considerably different. 

Similar to the argument for the previous part, we set  $j_0 = \lceil \frac{2}{4\beta_2 + d} \log_2 n \rceil$. We assume that we have data $\{ \mathbf{x}_i, y_i \}_{i =1}^{2n}$. We split it into two equal parts and use the second part to construct the estimator $\hat{g}$ of the design density $g$ introduced in Theorem \ref{thm:estimation_regression}.  Throughout the proof, $\E_{i,P}[\cdot]$ will denote the expectation with respect to the $i^{th}$ half of the sample, with the other half held fixed, under the distribution $P$. For $J_0 \leq l \leq j_0$, we construct the test statistics 
	\begin{align*}
		T_n(l) = \frac{1}{n(n-1)} \sum_{1\leq i \neq j\leq n} \frac{y_i }{ \hat{g}(\mathbf{x}_i)}  K_{W_l} (\mathbf{x}_i , \mathbf{x}_j ) \frac{y_j}{ \hat{g}(\mathbf{x}_j)}. 
	\end{align*}
	
	By Markov inequality, there exits a constant $C^*$ such that 
		\be 
		\P_P[ \| \hat{g} - g \|_2^2 > {C^*}^2 n^{- \frac{2\gamma}{2 \gamma +d }}] < \frac{\alpha}{4}. \label{eqn:cstar_def}
		\ee
		
		  We will condition on this event throughout this proof. The construction of the test depends on the following two lemmas. 
\begin{lemma}
\label{lemma:ustat_deviation_bound}
For $0< \alpha <1$, there exists $\zeta$ sufficiently large such that 
\be
\P_P\Big[ \forall \,\, J_0 \leq l \leq j_0, | T_n(l) - \| \Pi_{W_l} (f \frac{g}{\hat{g}}) \|_2^2 | \leq \zeta \sqrt{\frac{2^{(l+ j_0)d/2} }{n^2} + 2^{ld/4} \frac{ \| \Pi_{W_l}(f \frac{g}{\hat{g}})\|_2^2}{n}} \Big] \geq 1- 3\alpha/4. \nonumber 
\ee
\end{lemma}


\begin{lemma}
\label{lemma:signal}
\begin{itemize}
\item Under $H_0$, $\sup_{J_0 \leq l \leq j_0} \Big( \| \Pi_{W_l}(f \frac{g}{\hat{g}}) \|_2 - ( \frac{M}{2^{l\beta_1}} + \frac{C^*}{B_L'} n^{-\frac{\gamma}{2\gamma + d}}) \Big) \leq 0$ with probability at least $(1- \alpha/4)$. 

\item 
Let $\{ \tau_l : J_0 \leq  l \leq j_0 \}$ be a sequence of numbers satisfying $\sum_{l = J_0}^{j_0} \tau_l \leq \frac{3}{4} \sqrt{D} n^ {- \frac{2\beta_2}{4\beta_2 + d}}$. Then under $(f,g) \in H_1$, with probability at least $1- \alpha/4$, there exists $J_0 \leq l \leq j_0$ such that 
$\| \Pi_{W_l} ( f \frac{g}{\hat{g}} ) \|_2 \geq \frac{M}{2^{l\beta_1}} + \tau_l$. 
\end{itemize}
\end{lemma}
Before proving these two lemmas, we first complete the proof of the theorem assuming the validity of these two lemmas.

We consider the test $\Psi$ which rejects if at least one of the $T_n(l) > \tilde{C}_l$, where 
	\be
		\tilde{C}_l = \Big(\frac{M}{2^{l \beta_1}} + \frac{C^*}{B_L'} n^{- \frac{\gamma}{2 \gamma + d}} + \zeta \frac{2^{(l + j_0) d/8}}{\sqrt{n}} \Big)^2 \label{eqn:composite_cutoff}
	\ee
	for $\zeta$ suitably large, to be chosen appropriately.  We will use the following deviation bounds to control the Type I and II errors of this testing procedure.

We first control the Type I error of this procedure. Under $H_0$, with probability at least $1-\alpha$, for all $J_0 \leq l \leq j_0$,
\begin{align}
T_n(l) &\leq \| \Pi_{W_l} ( f \frac{g}{\hat{g}} ) \|_2^2  + \zeta \frac{ 2^{(l + j_0 ) d/4}}{n} + \zeta 2^{ld/8} \frac{ \| \Pi_{W_l} (f \frac{g}{\hat{g}}) \|_2 }{\sqrt{n}} \nonumber \\
&\leq \Big( \frac{M}{2^{l \beta_1}} + \frac{C^*}{B_L'} n^{- \frac{\gamma} {  2\gamma +d } } \Big)^2 + \zeta^2 \frac{2^{(l+ j_0)d/4}}{n} + 2 \zeta \frac{2^{(l+ j_0)d/8}}{\sqrt{n}} \Big( \frac{M}{2^{l \beta_1}} + \frac{C^*}{B_L'} n^{-\frac{\gamma}{ 2\gamma +d} } \Big) \nonumber \\
&\leq \Big(\frac{M}{2^{l\beta_1}} + \frac{C^*}{B_L'}n^{-\frac{\gamma}{2\gamma + d}} + \zeta \frac{2^{(l+ j_0) d/8}}{\sqrt{n}}\Big)^2 , \nonumber
\end{align}
where we assume that $\zeta>1$ without loss of generality. This controls the Type I error. 

To control the Type II error, we fix $(f,g) \in \Par(\beta_1,\beta_2, \gamma, \rho_n^2)$. Using Lemma \ref{lemma:signal}, there exits $J_0 \leq l \leq j_0$ such that 
\begin{align}
\| \Pi_{W_l}( f \frac{g}{\hat{g}} ) \|_2 \geq \frac{M}{2^{l\beta_1}} + \tau_l,  \nonumber 
\end{align}
where we choose $\tau_l = C_1 ( n^{- \frac{\gamma}{2\gamma + d} } + \frac{2^{(l+ j_0)d/8}}{\sqrt{n}})$. Thus with probability at least $(1- \alpha)$, we have,
\begin{align}
T_n(l) &\geq \| \Pi_{W_l} ( f \frac{g}{\hat{g}}) \|_2^2 - \zeta \sqrt{\frac{2^{(l+ j_0)d/2}}{n^2} + 2^{ld/4} \frac{ \| \Pi_{W_l}(f \frac{g}{\hat{g}})\|_2^2}{n} }. \nonumber \\
&\geq \| \Pi_{W_l}(f \frac{g}{\hat{g}}) \|_2 \Big( \| \Pi_{W_l}(f \frac{g}{\hat{g}}) \|_2 - \zeta \frac{2^{ld/8}}{\sqrt{n}} \Big) - \zeta \frac{2^{(l+j_0)d/4}}{n} \nonumber\\
&\geq \Big(\frac{M}{2^{l\beta_1}} + \tau_l  \Big) \Big( \frac{M}{2^{l \beta_1}} + \tau_l - \zeta \frac{2^{ld/8}}{\sqrt{n}} \Big) - \zeta \frac{2^{(l+j_0)d/4}}{n} \nonumber \\
&\geq \Big(\frac{M}{2^{l\beta_1}} + \tau_l  \Big) \Big( \frac{M}{2^{l \beta_1}} + \tau_l /2 \Big) - \zeta  \frac{2^{(l+j_0)d/4}}{n}, \nonumber 
\end{align}
where we choose $C_1 > 2 \zeta$. Now, choosing $C_1$ even larger, specifically $C_1^2 > 4 \zeta$, it follows that $\zeta \frac{2^{(l+j_0)d/4}}{n} \leq \tau_l^2/4$. Thus for some $ J_0 \leq l \leq j_0$, with probability at least $1 - \alpha$, 
\begin{align}
T_n(l) &\geq \Big(\frac{M}{2^{l\beta_1}} + \frac{\tau_l}{2} \Big)^2 \nonumber \\ 
&\geq \Big( \frac{M}{2^{l\beta_1}} + \frac{C^*}{B_L'} n^{- \frac{\gamma}{2\gamma + d}} + \zeta \frac{2^{(l+j_0)d/8}}{\sqrt{n}} \Big)^2 , \nonumber 
\end{align}
provided we choose $C_1 > 2 C^*/ B_L'$. This controls the Type II error of this test. 

The proof of the Theorem will now be completed with the proofs of Lemma \ref{lemma:ustat_deviation_bound} and Lemma \ref{lemma:signal} in the next two subsections.

\paragraph{Proof of Lemma \ref{lemma:ustat_deviation_bound}}
The proof of this lemma can indeed be completed by invoking Lemma \ref{lemma_ustat_tail_use}, which yields a much stronger control of the tail bound than demanded by Lemma \ref{lemma:ustat_deviation_bound}. However, for the sake of simplicity we provide simpler proof by simple union bound and Chebychev's inequality. The proof follows by an argument similar to \cite[Lemma 4.2]{carpentier2015regularity}. We have, for $J_0 \leq l \leq j_0$, 
\begin{align}
\E_{1,P}[T_n(l)] = \| \Pi_{W_l} ( f \frac{g}{\hat{g}} ) \|_2^2  \,\,\,\,\,\,\, \Var_{1,P}[T_n(l)] \leq C(S,B_L,B_U) \Big(\frac{\|\Pi_{W_l}(f \frac{g}{\hat{g}}) \|_2^2}{n} + \frac{2^{ld}}{n^2} \Big), \nonumber
\end{align}
The validity of the variance bound of the last display above, follows from Hoeffding's decomposition, boundedness of $f,g,\ghat$, and standard properties of compactly supported wavelet bases. Therefore, we have, using union bound and Chebychev's inequality, 
\begin{align}
&\P_P\Big[ \exists l, J_0 \leq l \leq j_0, | T_n(l) - \| \Pi_{W_l}(f \frac{g}{\hat{g}}) \|_2^2| > \zeta \sqrt{\frac{2^{(l+j_0)d/2}}{n^2} + 2^{ld/4} \frac{\|\Pi_{W_l}(f\frac{g}{\hat{g}}) \|_2^2}{n}} \Big] \nonumber \\
&\leq \sum_{l=J_0}^{j_0} \E_P\Big[\P_{1,P}\Big[ |T_n(l) - \| \Pi_{W_l} (f \frac{g}{\hat{g}}) \|_2^2| > \zeta \sqrt{\frac{2^{(l+j_0)d/2}}{n^2} + 2^{ld/4} \frac{\|\Pi_{W_l}(f\frac{g}{\hat{g}}) \|_2^2}{n}}\Big] \Big] \nonumber \\
&\leq \sum_{l=J_0}^{j_0} \E_P\Big [ \frac{\Var_{1,P}[T_n(l)] }{ \zeta^2 [\frac{2^{(l+j_0)d/2}}{n^2} + 2^{ld/4} \frac{\|\Pi_{W_l}(f\frac{g}{\hat{g}}) \|_2^2}{n}]} \Big] 
\leq \frac{C(S,B_L,B_U)}{\zeta^2} \sum_{l=J_0}^{j_0} [2^{- (j_0 - l )d/2}  + 2^{- ld /4} ] . \nonumber 
\end{align} 
The proof follows upon noting that $\sum_{l=J_0}^{j_0} 2^{- (j_0 -l ) d/2} \leq \frac{2^{d/2}}{2^{d/2}-1}$ and $\sum_{l=J_0}^{j_0} 2^{- ld/4} \leq \frac{ 2^{-J_0 d/4}}{1- 2^{-d/4}}$.

\paragraph{Proof of Lemma \ref{lemma:signal}} 
Let us consider $f \in \besov^{\beta_1}(M)$. Setting $\hat{\Delta} = \frac{g - \hat{g}}{\hat{g}}$, we have, for $J_0 \leq l \leq j_0$, 
\begin{align}
\| \Pi_{W_l}(f \frac{g}{\hat{g}} ) \|_2 \leq \| \Pi_{W_l}(f) \|_2 + \|\Pi_{W_l} ( f \hat{\Delta}) \|_2. \nonumber. 
\end{align}
For $f \in \besov^{\beta_1}(M)$, it follows from definition that $ \| \Pi_{W_l}(f) \|_2 \leq \frac{M}{2^{l \beta_1}}$. Recalling the definition of $C^*$ from \eqref{eqn:cstar_def},  the property of $\ghat$ from Theorem \ref{thm:estimation_regression}, and using the contraction property of the norm under projections, we have with probability at least $(1-\alpha/4)$, 

\begin{align}
\| \Pi_{W_l} (f \hat{\Delta} ) \|_2^2 \leq \| f \hat{\Delta} \|_2^2 \leq \Big(\frac{C^*}{B_L'}\Big)^2 n^{- \frac{2\gamma}{2\gamma +d}}. \nonumber
\end{align}
Combining, we get the desired result for functions $f \in \besov^{\beta_1}(M)$. 

Next,we consider functions $f \in \besov^{\beta_2}(M)$ such that $\| f - \besov^{\beta_1}(M) \|_2 > \rho_n$.  We first note that for any $h \in \besov^{\beta_1}(M)$, 
\begin{align}
\| \Pi_{V_{j_0}} ( f) - h \|_2 \geq \rho_n - \| f - \Pi_{V_{j_0}}(f) \|_2 \geq \frac{5}{6} \rho_n \nonumber 
\end{align}
if $D$ is chosen large enough. This implies that with probability at least $(1- \alpha/4)$, we have,
\begin{align}
\| \Pi_{V_{j_0}} ( f \frac{g}{\hat{g}} ) - h \|_2 \geq \frac{5}{6} \rho_n - \| \Pi_{V_{j_0}} ( f \hat{\Delta}) \|_2 \geq \frac{3}{4} \rho_n \nonumber
\end{align}
for $n$ sufficiently large, if $\gamma > 2 \beta_2$. Thus, if $\{\tau_l : J_0 \leq l \leq j_0 \}$ is a sequence of numbers 
such that $\frac{3}{4} \rho_n \geq \sum_{l = J_0}^{j_0} \tau_l$, following the argument of \cite[Lemma 4.1]{carpentier2015regularity}, it is easy to see that there 
exists $ J_0 \leq l \leq j_0$ such that $\| \Pi_{W_l} ( f \frac{g}{\hat{g}}) \|_2 \geq \frac{M}{2^{l\beta_1}} + \tau_l $. We choose 
\begin{align}
\tau_l = C_1 \Big( n ^{- \frac{\gamma}{2\gamma + d} } + \frac{2^{(l+j_0) d/8}}{\sqrt{n}} \Big) \nonumber, 
\end{align}
where $C_1$ will be chosen suitably. It is easy to see that for any chosen $C_1$,  $\sum_{l} \tau_l \leq \frac{3}{4} \rho_n$ can be enforced by choosing $D$ sufficiently large.

\subsection{Proof of Theorem \ref{theorem_adaptive_confidence_set}}
\label{sec:confidence_proof}
This proof idea is motivated by \cite{bull2013adaptive}. For $\beta\in [\beta_{\min},\beta_{\max}]$ and $2\beta_{\max}<\gamma<\gamma_{\max} $ consider $P=(f,g) \in \Par_n(M^*,M',\gamma) \cap \Par(\beta,\gamma)$.   Recall the finite grid $\{ \beta_1, \cdots, \beta_N\}$ used for the construction of the parameter spaces $\Par_n(M^*,M',\gamma)$.  We define $\F_n(M^*, j) = {\mbesov^{\beta_j}}(M, M^* \rho_n(\beta_j))$, $j = 1, \cdots, N-1$, $\F_n(M^*, N) = \besov^{\beta_N}(M)$. In addition, we set, for $j \in \{1, \cdots, N\}$, 
\begin{align}
\Par_n(j, M^*,M',\gamma) = \{ (h,g): h\in \F_n(M^*, j), 0<h< 1, g \in \besov^{\gamma}(M'), B_L < g < B_U, \int g(\bx) d\bx =1 \}. \nonumber 
\end{align}

Recall further the test $\Psi$ introduced in the proof of Theorem \ref{thm:testing} part \ref{thm:testing_composite}. Note that cut-off for the test, as in \eqref{eqn:composite_cutoff}, depends on the smoothness of $g$. However, a close inspection of the proof reveals that the only requirement on the smoothness of $g$ is that of being at least as large as twice the maximum smoothness of $f$. Since, our estimator $\ghat$ is an adaptive estimator of $g$, and $\gmin>2\beta_{\max}$, we can use $\gmin$ in the cut-off \eqref{eqn:composite_cutoff} for the test $\Psi$, maintaining the validity of the results. The test $\Psi$ with $\beta= \beta_j$ will be referred to as $\Psi(j)$. We first test the hypothesis $H_0: h \in \F_n(M^*,2)$ vs.  $H_1: h \in \F_n(M^*, 1)$ at level $\alpha/{4 N}$. If we reject $H_0$, we set $\hat{\beta}= \beta_1$ and stop. Otherwise we continue. At the $j^{th}$ step, $1<j < N-1$, we test $H_0: h \in \F_n(M^*, j+1)$ vs. $H_1: h \in \F_n(M^*,j)$ using the appropriate test $\Psi(j)$ at level $\alpha/{(4N)}$. If we reject $H_0$ at step $j$, we set $\hat{\beta}=\beta_j$ and stop. Otherwise we continue--- if none of the hypotheses are rejected, we set $\hat{\beta}= \beta_N$. 
This procedure determines the ``shell" in which $f$ belongs. Once this has been accomplished, we construct a confidence set using ideas introduced in \cite{robins2006adaptive}.

Without loss of generality, we assume we have data $\{\bx_i, y_i : 1\leq i \leq 3n\}$. We split the data into three equal parts--- the estimator $\fhat$ outlined in Theorem \ref{thm:estimation_regression} and $\hat{\beta}$ described above are constructed from the first, while the adaptive estimator of the design density $\ghat$ introduced in Theorem \ref{thm:estimation_regression} is constructed from the second part. We condition on the events $\{\hat{f}\in \besov^{\beta}(C')\}$ and $\{\hat{g}\in \besov^{\gamma}(C')\}$ which happen with probability at least $1-r_n$ (for $C'$ large enough depending on $M,M',B_U,B_L,\gmax$) uniformly over $\Par_n(M^*,M',\gamma)\cap \Par(\beta,\gamma)$, for some vanishing sequence $r_n$. Finally, we set $j_1 = \lceil \frac{2}{4\hat{\beta} +d} \log_2(n) \rceil$. Using the data $\{(\bx_i, y_i): 1\leq i \leq n\}$, we construct the following U-statistic. 
\be
\hat{U}_n = \frac{1}{n(n-1)} \sum_{2n+1\leq i_1 \neq i_2\leq 3n} \frac{(y_{i_1} - \hat{f}(\bx_{i_1}) )}{{\ghat({\bx_{i_1}})}} K_{V_{j_1}}\left(\bx_{i_1},\bx_{i_2}\right) \frac{(y_{i_2} - \hat{f} (\bx_{i_2}))}{{\ghat({\bx_{i_2}})}}. 
\ee
For any $h \in L^2$, we define
$\tau_n^2(h) = \frac{C_1}{n} \| h - \fhat_n \|_2^2 + \frac{C_2 2^{j_1}}{n(n-1)}$, 
for constants $C_1, C_2$ to be chosen later. Finally, we define the set 
\be
C_n(\beta) = \left\{ h : \| h - \hat{f} \|_2^2 \leq \hat{U}_n + C(M,B_L,B_U)\left(n^{-\frac{4\beta}{4\beta+d}}+n^{-\frac{\beta}{2\beta+d}}n^{-\frac{\gamma_{\min}}{2\gamma_{\min}+d}}\right) + z(\alpha) \tau_n(h)\right\},
\ee
with $z(\alpha) \geq 1/{\alpha}$. We will show that the set $C_n(\hat{\beta})$ is a confidence set with the desired properties. 
Throughout the rest of the proof $\E_{P, S}[\cdot]$ for $P\in \Par_n(M^*,M',\gamma)$ and $S\subset \{1,2,3\}$ will denote expectation under the distribution $P$ conditional on the subset of the data corresponding to the subset $S$. 

Let $i_0 = i_0(f) \in \{1,\cdots, N\}$ denote the unique index such that $f \in \F_n(M^*, i_0)$. We prove that uniformly over $P\in \Par_n(M^*,M',\gamma)\cap \Par(\beta,\gamma)$,  $\P_P(\hat{\beta}\neq \beta_{i_0}) \leq \alpha/2$. Indeed,  $\hat{\beta} < \beta_{i_0}$ implies that one of the test $\Psi(j)$, $j=1, \cdots, i_0-1$ has rejected the true null hypothesis. Thus 
\be \sup_{P \in \Par_n(i_0, M^*,M',\gamma)} \P_P[\hat{\beta} < \beta_{i_0}] \leq \sum_{i < i_0} \sup_{P\in \Par_n(i_0, M^*,M',\gamma)}\E_P[\Psi(i)] < \frac{\alpha}{4}.\ee 
Similarly, $\hat{\beta} > \beta_{i_0}$ essentially implies that one of the tests $\Psi(i)$, $i> i_0$ fails to reject the null hypothesis. Therefore
\be \sup_{P\in \Par_n(i_0, M^*,M',\gamma)} \P_P[\hat{\beta} > \beta_{i_0}] \leq \sum_{i > i_0} \sup_{P \in \Par_n(i_0, M^*,M',\gamma) }\E_P[1- \Psi(i)] \leq \frac{\alpha}{4}.\ee
Combining, we have, $\sup_{P \in \Par_n(i_0, M^*,M',\gamma)}\P_P[\hat{\beta} \neq \beta_{i_0}] \leq \frac{\alpha}{2}$. Now, we have,
\begin{align}
\P_P[ f \in C_n(\hat{\beta})] \geq \P_P[ f \in C_n(\beta_{i_0})] - \frac{\alpha}{2}. \nonumber 
\end{align}
Thus honesty of the confidence set follows provided we establish that $\P_P[f \in \hat{C}_n(\beta_{i_0})] \geq 1- \alpha/2 $ uniformly on $\Par_n(i_0, M^*,M', \gamma)$. To this end, we note that 
setting $\deltahat=\frac{\ghat-g}{g}$, we have that for a  deterministic constant $C(M,B_L, B_U)$,
\begin{align*}
	\ & \E_{P,\{2,3\}}[ \hat{U}_n] = \| \Pi_{V_{j_1}} ( f  - \hat{f}_n)\fracg \|_2^2  \nonumber \\
	&= \left\|\Pi_{V_{j_1}}\left(\delf\right)\right\|_2^2+\left\|\Pi_{V_{j_1}}\left(\left(\delf\right)\deltahat\right)\right\|_2^2 \\&+ 2\left\langle \Pi_{V_{j_1}}\left(\delf\right), \Pi_{V_{j_1}}\left(\left(\delf\right)\deltahat\right) \right\rangle\nonumber\\
	&=\|\delf\|_2^2-\left\|\Pi_{V_{j_1}^{\perp}}\left(\delf\right)\right\|_2^2+\left\|\Pi_{V_{j_1}}\left(\left(\delf\right)\deltahat\right)\right\|_2^2\\&+2\left\langle \Pi_{V_{j_1}}\left(\delf\right), \Pi_{V_{j_1}}\left(\left(\delf\right)\deltahat\right) \right\rangle \nonumber\\
	& \geq \|\delf\|_2^2-C(M,M', \beta_{\max})n^{-\frac{4\beta_{i_0}}{4\beta_{i_0}+d}}\\&-2\left\|\Pi_{V_{j_1}}\left(\delf\right)\right\|_2\left\|\Pi_{V_{j_1}}\left(\left(\delf\right)\deltahat\right)\right\|_2 \nonumber\\
	& \geq \|\delf\|_2^2-C(M, M',\bmax, B_L, B_U)\Big(n^{-\frac{4\beta_{i_0}}{4\beta_{i_0}+d}}+n^{-\frac{\beta_{i_0}}{2\beta_{i_0}+d}}n^{-\frac{\gamma}{2\gamma+d}}\Big). \nonumber
\end{align*}

Further, we have, using Hoeffding decomposition conditional on samples $\{2,3\}$, 
\begin{align}
	\hat{U}_n - E_{P,\{2,3\}}[\hat{U}_n]  = L + R&, \nonumber\\
	L = \frac{2}{n} \sum_{i=2n+1}^{3n} \sum_{k \in \mathcal{Z}_{j_1}} \sum_{ v \in \{0, 1\}^d} \Big[\frac{(y_i - \hat{f}_n (\bx_i) )}{\ghat(\bx_i)}&\psi_{j_1, k}^{v} (\bx_i) - \langle (\delf)\fracg , \psi _{j_1, k }^{v} \rangle \Big] \langle (\delf)\fracg , \psi _{j_1, k }^{v} \rangle, \nonumber\\
	R = \frac{1}{n(n-1)} \sum\limits_{2n+1 \leq i_1 \neq i_2 \leq 3n} \sum_{k \in \mathcal{Z}_{j_1}} \sum_{ v \in \{0, 1\}^d} &\Big[\frac{(y_{i_1} - \hat{f}_n (\bx_{i_1}) )}{\ghat(\bx_{i_1})}\psi_{j_1, k}^{v} (\bx_{i_1}) - \langle (\delf)\fracg , \psi _{j_1, k }^{v} \rangle \Big] \nonumber\\
	\times & \Big[\frac{(y_{i_2} - \hat{f}_n (\bx_{i_2}) )}{\ghat(\bx_{i_2})}\psi_{j_1, k}^{v} (\bx_{i_2}) - \langle (\delf)\fracg , \psi _{j_1, k }^{v} \rangle \Big].\nonumber
\end{align}
Using the orthogonality of the linear and non-linear term in Hoeffding's decomposition, we can bound the variance of $\hat{U}_n$ by the sum of the variances of $L$ and $R$. The variance of the linear term may be bounded by the second moment and using the boundedness of $f,\hat{f}_n,\ghat_n$, we have that 
\begin{align*}
	\var_{P,\{2,3\}}[L] \leq \frac{C(S,B_L,B_U)}{n} \| \Pi_{V_{j_1}} ( f -\hat{f}_n )\fracg \|_2^2 .
\end{align*}
By a proof similar to that of controlling $\Lambda_1$ in Lemma \ref{lemma_lamda_estimates}, we have that for a deterministic constant $C(\beta_{\max},\gamma_{\max}, M, B_L, B_U)$
\begin{align*}
	\var_{P, \{2,3\}}[R] \leq \frac{C(S, B_L, B_U) 2^{j_1}}{n(n-1)}.  
\end{align*}
Finally, we set 
\begin{align*}
	\tau_n(f)^2 =  \frac{C(S, B_L, B_U)}{n} \|  ( f -\hat{f}_n ) \|_2^2 + \frac{C(S, B_L, B_U) 2^{j_1}}{n(n-1)}. 
\end{align*}

By an application of Chebychev inequality, we have, 
\begin{align*}
	\P_{P,\{2,3\}}[ | \hat{U}_n - \E_{P,\{2,3\}}[ \hat{U}_n] | > C\tau_n(f)] & \leq \frac{\Var_{P, \{2,3\}}[ \hat{U}_n] } { C^2 \tau_n(f)^2} \leq \frac{1}{C^2}. 
\end{align*}
Thus for $C$ chosen appropriately, the above probability may be controlled at any pre-specified level $\alpha/2$.

Based on our construction, we have, 
	\begin{align*}
		& \P_{P}\left( f \in C_n(\beta_{i_0}) \right)  \nonumber \\
		&= \P_P\Big( \| f - \hat{f} \|_2^2 \leq \hat{U}_n + C(M, M',\bmax, B_L, B_U)\Big(n^{-\frac{4\beta_{i_0}}{4\beta_{i_0}+d}}+n^{-\frac{\beta_{i_0}}{2\beta_{i_0}+d}}n^{-\frac{\gamma_{\min}}{2\gamma_{\min}+d}}\Big) + z(\alpha) \tau_n(f) \Big) \nonumber\\
		&\geq \P_P [ | \hat{U}_n - \E_{P,\{2,3\}}[\hat{U}_n] | \leq z(\alpha) \tau_n(f) ] \geq (1- \frac{\alpha}{2}). \nonumber
	\end{align*}
	
	Finally, we establish that the $L^2$ diameter of this set adapts to the underlying smoothness. Assume $P\in \Par_n(M^*,M',\gamma)\cap \Par(\beta,\gamma)$ and the following calculations are uniform over this parameter space. The deterministic terms in the diameter term are respectively of the order $n^{-\frac{2\beta_{i_0}}{4\beta_{i_0} + d}} = o(n^{- \frac{\beta}{2\beta + d}})$ (as $\beta < \beta_{i_0+1} < 2 \beta_{i_0}$) and $n^{-\frac{\beta_{i_0}}{2\beta_{i_0}+d}}n^{-\frac{\gamma_{\min}}{2\gamma_{\min}+d}}$ which, by some tedious algebra, is also  $o(n^{- \frac{2\beta}{2\beta + d}})$ since $\beta < \beta_{i_0+1} < 2 \beta_{i_0}$, $\gamma_{\min}>2\beta_{i_0+1}$. The random part of $\tau_n(f)^2$ is also $o_P(n^{- \frac{2\beta}{2\beta + d}})$ as $\hat{f}_n$ is an adaptive estimator and $\|\fracg\|_{\infty}\leq \frac{ B_U}{B_L'(B_L)}=C(B_U,B_L)$. Finally, the leading term for the diameter is contributed by 
	\be
		\E_P[ \hat{U}_n] = \E_P\Big[ \| \Pi_{V_{j_1}} (f - \hat{f}_n)\fracg \|_2^2 \Big] \leq C(B_U,B_L) \|\delf\|_2^2,
	\ee
	which is $O_P(n^{- \frac{2\beta}{2\beta + d}})$ as $\hat{f}_n$ is adaptive. This completes the proof.

\subsection{Proof of Theorem \ref{thm:estimation_regression}}
\label{section:estimation_proof}

\subsubsection{Proof of Part \ref{thm:estimation_part1}}

Without loss of generality assume that we have data $\{ \mathbf{x}_i, y_i \}_{i =1}^{2n}$. We split it into two equal parts and use the second part to construct the estimator $\hat{g}$ of the design density $g$. Throughout the proof, $\E_{i,P}[\cdot]$ will denote the expectation with respect to the $i^{th}$ half of the sample, with the other half held fixed, under the distribution $P$. Throughout we choose the regularity of our wavelet bases to be larger than $\gamma_{\max}$ for the desired approximation and moment properties to hold. As a result our constants depend on $\gamma_{\max}$.

Let $2^{\jmin d}=\lfloor n^{\frac{1}{2\betamax/d+1}} \rfloor$, $2^{\jmax d}=\lfloor n^{\frac{1}{2\betamin/d+1}} \rfloor$, $2^{\lmin d}=\lfloor n^{\frac{1}{2\gmax/d+1}} \rfloor$, and $2^{\lmax d}=\lfloor n^{\frac{1}{2\gmin/d+1}} \rfloor$ and define $\mathcal{T}_1=[\jmin,\jmax]\cap \mathbb{N}$ and $\mathcal{T}_2=[\lmin,\lmax]\cap \mathbb{N}$. Let $\ghat_l=\frac{1}{n}\sum_{i=n+1}^{2n}K_{V_l}\left(\bx_i,x\right)$.

Now, let
\be
\lhat=\min\Big\{j\in \mathcal{T}_2: \ \|\ghat_j-\ghat_{l}\|_2\leq C^*\sqrt{\frac{2^{ld}}{n}}, \ \forall l \in \mathcal{T}_2 \ \text{s.t.} \ l \geq j \Big\}.
\ee
where $C^*$ is a constant  ({depending on $\gmax, B_U$}) that can be determined from the proof hereafter. Thereafter, consider the Lepski-type estimator $\gtilde:=\ghat_{\lhat}$ \citep{lepskii1991problem,lepskii1992asymptotically}. The following lemma states the mean squared properties of $\gtilde$.
\begin{lemma}\label{lemma_adaptive_gtilde}(Theorem 2 of \cite{bull2013adaptive}) For any $\gmin \leq \gamma\leq \gmax$,
$$\sup\limits_{P \in \mathcal{P}(\beta,\gamma)} \E_P\|\gtilde-g\|_2^2 \leq  \left(C\right)^{\frac{2d}{2\gamma+d}}n^{-\frac{2\gamma}{2\gamma+d}},$$ 
with a large enough positive constant $C$ depending on $M$ and $B_U$.	
	
\end{lemma}
Although the proof of Lemma \ref{lemma_adaptive_gtilde} can be found in \cite{bull2013adaptive}, since we need certain steps of the proof in our subsequent analysis, we provide the proof again in the Appendix \ref{sec:proof_lemma_adaptive_gtilde}.

	Now we prove that $\liminf_{n \to \infty} \inf_{P \in \Par(\beta, \gamma)}\P_P[ \gtilde \in B_{2,\infty}^{\gamma}(C)] = 1$ for large enough constant $C$. Indeed, for any $C>0$ and $\lprime \geq J_0$, (letting for any $h \in L_2[0,1]^d$, $\|\langle h,\bpsi_{\lprime,\cdot}\rangle\|_2$ be the vector $L_2$ norm of the vector $\left(\langle h,\bpsi^v_{\lprime,\kprime}\rangle: \kprime \in \mathcal{Z}_{\lprime},v \in \left\{0,1\right\}^d-\{0\}^d\right)$. We have, 
	\be 
	\ & \P_P\left(2^{\lprime\gamma}\|\langle \gtilde, \bpsi_{\lprime,\cdot}\rangle \|_2>C\right)
	 = \sum_{l=\lmin}^{\lmax}\P_P\left(2^{\lprime\gamma}\|\langle \ghat_l, \bpsi_{\lprime,\cdot}\rangle \|_2>C, \lhat=l\right)\I\left(\lprime \leq l\right)\\
	&= \sum_{l=\lmin}^{\lstar}\P_P\left(2^{\lprime\gamma}\|\langle \ghat_l, \bpsi_{\lprime,\cdot}\rangle \|_2>C, \lhat=l\right)\I\left(\lprime \leq l\right)+\sum_{l=\lstar+1}^{\lmax}\P_P\left(2^{\lprime\gamma}\|\langle \ghat_l, \bpsi_{\lprime,\cdot}\rangle \|_2>C, \lhat=l\right)\I\left(\lprime \leq l\right)\\
	& \leq \sum_{l=\lmin}^{\lstar}\P_P\left(2^{\lprime\gamma}\|\langle \ghat_l, \bpsi_{\lprime,\cdot}\rangle \|_2>C\right)\I\left(\lprime \leq l\right)+\sum_{l=\lstar+1}^{\lmax}\P_P\left( \lhat=l\right)\I\left(\lprime \leq l\right)\\
	& \leq \sum_{l=\lmin}^{\lstar}\P_P\left(2^{\lprime\gamma}\|\langle \ghat_l, \bpsi_{\lprime,\cdot}\rangle \|_2>C\right)\I\left(\lprime \leq l\right)+\sum_{l>l^*}2e^{-C'2^{ld/2}}\I\left(\lprime \leq l\right)\label{eqn:ghat_fourier_lprime}
	\ee

where the last inequality follows from \eqref{eqn:wrongchoiceproblepski_g} for a suitable $C'$ (depending on $B_U$ and the wavelet basis choice). Now,
\be 
\ & \P_P\left(2^{\lprime\gamma}\|\langle \ghat_l, \bpsi_{\lprime,\cdot}\rangle \|_2>C\right)\\
&\leq \P_P\left(2^{\lprime\gamma}\|\langle \ghat_l, \bpsi_{\lprime,\cdot}\rangle-\E_P\left(\langle \ghat_l, \bpsi_{\lprime,\cdot}\rangle\right) \|_2>C/2\right)+\P_P\left(2^{\lprime\gamma}\|\E_P\left(\langle \ghat_l, \bpsi_{\lprime,\cdot}\rangle\right) \|_2>C/2\right)\\
&=\P_P\left(2^{\lprime\gamma}\|\langle \ghat_l, \bpsi_{\lprime,\cdot}\rangle-\E_P\left(\langle \ghat_l, \bpsi_{\lprime,\cdot}\rangle\right) \|_2>C/2\right)
\ee
if $C>2M'$. Therefore, from \eqref{eqn:ghat_fourier_lprime}, one has for any $C>2M'$, 
\be 
	\  \P_P\left(2^{\lprime\gamma}\|\langle \ghat, \bpsi_{\lprime,\cdot}\rangle \|_2>C\right)
	&\leq \sum_{l=\lmin}^{\lstar}\P_P\left(2^{\lprime\gamma}\|\langle \ghat_l, \bpsi_{\lprime,\cdot}\rangle-\E_P\left(\langle \ghat_l, \bpsi_{\lprime,\cdot}\rangle\right) \|_2>C/2\right)\I\left(\lprime \leq l\right)\\
	&+\sum_{l>l^*}2e^{-C'2^{ld/2}}\I\left(\lprime \leq l\right). \label{eqn:ghat_fourier_lprime_second}
	\ee
It remains to control $\|\langle \ghat_l, \bpsi_{\lprime,\cdot}\rangle-\E_P\left(\langle \ghat_l, \bpsi_{\lprime,\cdot}\rangle\right) \|_2$ appropriately. To this end, note that when $\lprime \leq l$, 
\be
\ & \|\langle \ghat_l, \bpsi_{\lprime,\cdot}\rangle-\E_P\left(\langle \ghat_l, \bpsi_{\lprime,\cdot}\rangle\right) \|_2^2
= \frac{1}{n^2}\sum_{i=n+1}^{2n}\sum\limits_{\kprime, v}\left(\bpsi_{\lprime,\kprime}^v(\bx_i)-\E_P\left(\bpsi_{\lprime,\kprime}^v(\bx_i)\right)\right)^2\\&+
\frac{1}{n^2}\sum\limits_{n+1\leq i_1\neq i_2\leq 2n}\sum\limits_{\kprime, v}\left(\bpsi_{\lprime,\kprime}^v(\bx_{i_1})-\E_P\left(\bpsi_{\lprime,\kprime}^v(\bx_{i_1})\right)\right)\left(\bpsi_{\lprime,\kprime}^v(\bx_{i_2})-\E_P\left(\bpsi_{\lprime,\kprime}^v(\bx_{i_2})\right)\right).
\ee
Note that the second term of the above summand is a type U-statistics of order $2$ analyzed in Lemma \ref{lemma_ustat_tail_use}. We make use of this fact below. 
\be
\ &\P_P\left(2^{2\lprime\gamma}\|\langle \ghat_l, \bpsi_{\lprime,\cdot}\rangle-\E_P\left(\langle \ghat_l, \bpsi_{\lprime,\cdot}\rangle\right) \|_2^2>C^2/4\right)
 \\&\leq \P_P\Big(\frac{1}{n^2}\sum_{i=n+1}^{2n}\sum\limits_{\kprime, v}\Big(\bpsi_{\lprime,\kprime}^v(\bx_i)-\E_P\left(\bpsi_{\lprime,\kprime}^v(\bx_i)\right)\Big)^2>\frac{C^2/8}{2^{2\lprime\gamma}}\Big)\\
&+\P_P\Big(\Big|\frac{1}{n^2}\sum\limits_{n+1\leq i_1\neq i_2\leq 2n}\sum\limits_{\kprime, v}\left(\bpsi_{\lprime,\kprime}^v(\bx_{i_1})-\E_P\left(\bpsi_{\lprime,\kprime}^v(\bx_{i_1})\right)\Big)\Big(\bpsi_{\lprime,\kprime}^v(\bx_{i_2})-\E_P\left(\bpsi_{\lprime,\kprime}^v(\bx_{i_2})\right)\right)\Big|>\frac{C^2/8}{2^{2\lprime\gamma}}\Big)\\
&=I+II.
\ee
To control $I$ note that for any fixed $\bx \in [0,1]^d$
$$\sum\limits_{\kprime, v}\left(\bpsi_{\lprime,\kprime}^v(\bx)-\E_P\left(\bpsi_{\lprime,\kprime}^v(\bx)\right)\right)^2\leq C(\bpsi_{0,0}^0, \bpsi_{0,0}^1,\gmax)2^{\lprime d},$$
and therefore 
$$\E_P\sum\limits_{\kprime, v}\left(\bpsi_{\lprime,\kprime}^v(\bx)-\E_P\left(\bpsi_{\lprime,\kprime}^v(\bx)\right)\right)^2 \leq C(\bpsi_{0,0}^0, \bpsi_{0,0}^1,\gmax)2^{\lprime d}.$$
Therefore by Hoeffding's Inequality, 
\be 
I&\leq \P_P\Big(\frac{1}{n}\sum_{i=n+1}^{2n}\sum\limits_{\kprime, v}\Big(\bpsi_{\lprime,\kprime}^v(\bx_i)-\E_P\left(\bpsi_{\lprime,\kprime}^v(\bx_i)\right)\Big)^2>\frac{nC^2/8}{2^{2\lprime\gamma}}\Big)\\
& \leq 2e^{-C(\bpsi_{0,0}^0, \bpsi_{0,0}^1,\gmax)\frac{n}{2^{2\lprime d}}\left(\frac{nC^2/8}{2^{2\lprime\gamma}}\right)^2 }.
\ee
Finally, arguing similar to Lemma \ref{lemma_ustat_tail_use} we also have that for a constant $C(B_U,\gmax)$
\be
II &\leq e^{-\frac{C t(\lprime)^2}{a_1(\lprime)^2}}+e^{-\frac{C t(\lprime)}{a_2(\lprime)}}+e^{-\frac{C \sqrt{t(\lprime)}}{\sqrt{a_3(\lprime)}}} 
\ee

where $t(\lprime)=\frac{C^2/8}{2^{2\lprime\gamma}}$, $a_1(\lprime)= \frac{1}{n-1}2^{\frac{\lprime d}{2}}$, $a_2(\lprime)=\frac{1}{n-1}\Big(\sqrt{\frac{2^{\lprime d}}{n}}+1\Big)$, $a_3(\lprime)=\frac{1}{n-1}\Big(\sqrt{\frac{2^{\lprime d}}{n}}+\frac{2^{\lprime d}}{n}\Big)$.

Therefore, for $C>2M'$
\be
& \sum_{\lprime\geq J_0}\P_P\left(2^{\lprime\gamma}\|\langle \gtilde, \bpsi_{\lprime,\cdot}\rangle \|_2>C\right) \nonumber\\
 & \leq \sum_{\lprime\geq J_0}\sum_{l=\lmin}^{\lstar}\P_P\left(2^{\lprime\gamma}\|\langle \ghat_l, \bpsi_{\lprime,\cdot}\rangle-\E_P\left(\langle \ghat_l, \bpsi_{\lprime,\cdot}\rangle\right) \|_2>C/2\right)\I\left(\lprime \leq l\right) +\sum_{\lprime\geq J_0}\sum_{l>l^*}2e^{-C'2^{ld/2}}\I\left(\lprime \leq l\right)\\
& \leq \sum_{l=\lmin}^{\lstar} \sum_{\lprime=J_0}^{l}2e^{-C(\bpsi_{0,0}^0, \bpsi_{0,0}^1,\gmax)\frac{n}{2^{2\lprime d}}\left(\frac{nC^2/8}{2^{2\lprime\gamma}}\right)^2 }  \nonumber\\
&+	\sum_{l=\lmin}^{\lstar} \sum_{\lprime=J_0}^{l}\Big(e^{-\frac{C t(\lprime)^2}{a_1(\lprime)^2}}+e^{-\frac{C t(\lprime)}{a_2(\lprime)}} 
+e^{-\frac{C \sqrt{t(\lprime)}}{\sqrt{a_3(\lprime)}}} \Big)+\sum_{l=\lmin}^{\lstar} \sum_{\lprime=J_0}^{l}2e^{-C'2^{ld/2}}.
\label{eqn:gtilde_smoothness_prob}
\ee
Some tedious calculations now show that the last term in the display above converges to $0$ uniformly in $P \in \Par(\beta,\gamma)$ as $n \rightarrow \infty$. This, along with the definition of $\besov^{\gamma}(C)$, completes the proof of $\liminf_{n \to \infty} \inf_{P \in \Par(\beta, \gamma)}\P_P[ \gtilde \in B_{2,\infty}^{\gamma}(C)] = 1$ for sufficiently large constant $C$ depending on $(M',B_U,\gmax)$.
	
	However this $\gtilde$ does not satisfy the desired point-wise bounds. To achieve this let $\psi$ be a $C^{\infty}$ function such that $\psi(x) |_{[B_L, B_U]} \equiv x$ while $\frac{B_L}{2} \leq \psi(\bx) \leq 2B_U$ for all $x$. Finally, consider the estimator $\hat{g}(\bx) = \psi(\tilde{g}(\bx))$. We note that $(g(\bx) - \hat{g}(\bx))^2 \leq (g(\bx) - \tilde{g}(\bx))^2$--- thus $\hat{g}$ is adaptive to the smoothness of the design density. The boundedness of $\ghat$ follows immediately from the construction. Finally, we wish to show that almost surely, the constructed estimator belongs to the Besov space with the same smoothness, possibly of a different radius. This is captured by the next lemma. The proof is deferred to Section \ref{proof_density_truncation}. 
	
	\begin{lemma}
		\label{lemma:density_truncation}
		For all $h \in \besovb(M)$, $\psi(h) \in \besovb(C(M, \beta))$, where $C(M,\beta)$ is a universal constant dependent only on $M,\beta$ and independent of $h \in \besovb(M)$. 
	\end{lemma} 
	
	The lower bound of the minimax estimation error follows in our case by the results of \cite{bull2013adaptive}, by setting $f\equiv 0$ in the prior used for the construction of the lower bound.

	 \subsubsection{Proof of Part \ref{thm:estimation_part2}}
	 For the construction of $\fhat$, construct the estimator $\hat{g}$ of the design density $g$ as above from second part of the sample and let $\hat{f}_j(\bx)=\frac{1}{n}\sum\limits_{i=1}^n \frac{y_i}{\ghat(\bx_i)}K_{V_j}\left(\bx_i, \bx\right)$. 
	Now, let
	\be
	\jhat=\min\Big\{j\in \mathcal{T}_1: \ \|\hat{f}_j-\hat{f}_{l}\|_2\leq C^{**}\sqrt{\frac{2^{ld}}{n}}, \ \forall l \in \mathcal{T}_1 \ \text{s.t.} \ l \geq j \Big\}.
	\ee
	where $C^{**}$ is a suitable constant (depending on $B_U,B_L,\gamma_{\max}$) to be decided later.  Thereafter, consider the estimator $\tilde{f}:=\hat{f}_{\jhat}$. 	
	
	Let $\jstar=\min\left\{j:C_{1*} 2^{-j\beta}\leq C_{2*}\sqrt{\frac{2^{jd}}{n}} \right\}$,and note that for any $\bx \in [0,1]^d$, 
	\be
	\int |\E_{P,1}\left(\fhat_j(\bx)\right)-f(\bx)|^2 d\bx
	&=\int \Big|\Pi\Big(f\frac{g}{\ghat}|V_j\Big)(\bx)-f(\bx)\Big|^2 d\bx\\
	&=\int \Big|\Pi\Big(f\Big(\frac{g}{\ghat}-1\Big)|V_j\Big)(\bx)-\Pi\left(f|V_j^{\perp}\right)(\bx)\Big|_2^2 d\bx\\
	&=\int \left|\Pi\left(f\left(\frac{g}{\ghat}-1\right)|V_j\right)(\bx)\right|^2d\bx+\int \left|\Pi\left(f|V_j^{\perp}\right)(\bx)\right|_2^2 d\bx\\
	&=\Big\|\Pi\Big(f\Big(\frac{g}{\ghat}-1\Big)|V_j\Big)\Big\|_2^2+\|\Pi\left(f|V_j^{\perp}\right)\|_2^2\\
	& \leq \Big\|f\left(\frac{g}{\ghat}-1\right)\Big\|_2^2+C_1^2 M^2 2^{-j\beta}.
	\ee
	Therefore, 
	\be 
	\  \E_{P,2}\int |\E_{P,1}\left(\fhat_j(\bx)\right)-f(\bx)|^2 d\bx
	 &\leq \E_{P,2}\Big\|f\Big(\frac{g}{\ghat}-1\Big)\Big\|_2^2+C_1^2 M^2 2^{-j\beta}\\
	& \leq \Big(\frac{B_U}{B_L^{'}}\Big)^2\left(C(M',B_U)\right)^{\frac{2}{2\gamma+d}}n^{-\frac{2\gamma}{2\gamma+d}}+C_1^2 M^2 2^{-j\beta}. 
	 \label{eqn:biasdecomposition_initial}
	\ee
     Since $\gamma_{\min}>\beta_{\max}$, we have from the definition of $\jstar$
	 \eqref{eqn:biasdecomposition_initial} that there exists a constant $C_{1*}$ depending on $M,M',B_U,B_L,\gmax$ such that 
	\be
	\E_{P,2}\int |\E_{P,1}\left(\fhat_{\jstar}(\bx)\right)-f(\bx)|^2 d\bx& \leq C_{1*}^2 2^{-2\jstar\beta} . \label{eqn:qbias_lepski_regression}
	\ee

	Also by Rosenthal's (Lemma \ref{lemma_linear_tail_bound}) and Jensen's Inequality, there exists a constant $C(q)$ for $q\geq 2$ such that
	\be
	\ & \E_{P,1}\left(|\fhat_j(\bx)-\E_{P,1}\left(\fhat_j(\bx)\right)|^q\right)\\
	&\leq \frac{C(q)}{n^q}\Big[\sum_{i=n+1}^{2n}\E_{P,1}\Big(\Big|\frac{y_i}{\ghat(\bx_i)}K_{V_j}(\bx_i,\bx)\Big|^q\Big)+\Big(\sum_{i=n+1}^{2n}\E_{P,1}\Big(\Big|\frac{y_i}{\ghat(\bx_i)}K_{V_j}(\bx_i,\bx)\Big|^2\Big)\Big)^{q/2}\Big]\\
	& \leq \frac{C_{2*}^q/2}{n^q}\times \left[n\left(2^{jd}\right)^{q-1}+n^{q/2}\left(2^{jd}\right)^{q/2}\right],\label{eqn:fhat_rosenthal}
	\ee
	where the last inequality in the above display follows by using standard facts about compactly supported wavelet basis having regularity larger than $\gamma_{\max}$ \citep{hardle1998wavelets} and the fact that the constructed $\ghat$ from the second half of the sample lies point-wise in $\left[\frac{B_L}{2},2B_U\right]$. The constant $C_{2*}$ therefore depends on $q$, the wavelet basis used, $B_U$ and $B_L$. Therefore, by the choice of $j \in \mathcal{T}_1$, we have that for all $\bx \in [0,1]^d$, 
	\be
	\E_{P,1}\left(|\fhat_j(\bx)-\E_{P,1}\left(\fhat_j(\bx)\right)|^q\right)&\leq C_{2*}^q\left(\frac{2^{jd}}{n}\right)^{q/2}.\label{eqn:qvariance_lepski_regression}
	\ee
	Therefore, using \eqref{eqn:qbias_lepski_regression} and \eqref{eqn:qvariance_lepski_regression}, we have the following bias-variance decomposition bound.
	\be 
	\  \E_P\left(\|\fhat_{\jstar}-f\|_2^2\right)&=\E_{P,2}\int\E_{P,1}\left(|\fhat_{\jstar}(\bx)-f(\bx)|^2\right)d\bx\\
	&= \E_{P,2}\Big[\int\E_{P,1}\left(|\fhat_{\jstar}(\bx)-\E_{P,1}\left(\fhat_{\jstar}(\bx)\right)|^2\right)d\bx+\int\E_{P,1}\left(|\E_{P,1}\left(\fhat_{\jstar}(\bx)\right)-f(\bx)|^2\right)d\bx\Big]\\
	& \leq C_{1*}^2  2^{-2\jstar\beta}+C_{2*}^2\Big(\frac{2^{\jstar d}}{n}\Big)\leq 2^{d+1}(C_{1*}^2+C_{2*}^2)n^{-\frac{2\beta}{2\beta+d}}.
	\ee

		
		 Therefore, by definition of $\jhat$ and $\jstar$,
		\be
		\ \E_P\left(\|\ftilde-f\|_2^2\I\left(\jhat \leq \jstar\right)\right)
		&\leq  2\E_{P}\left(\|\ftilde-\fhat_{\jstar}\|_2^2\I\left(\jhat \leq \jstar\right)\right)+2 \E_{P}\left(\|\fhat_{\jstar}-f\|_2^2\right)\\
		&\leq 2((C^{**})^2+2^{d+1}(C_{1*}^2+C_{2*}^2))n^{-\frac{2\beta}{2\beta+d}}. \label{eqn:lessthanlstar_regression}
		\ee
	By Cauchy-Schwarz inequality,
		\be
		\E_P\left(\|\ftilde-f\|_2^2\I\left(\jhat > \jstar\right)\right)& \leq \sum\limits_{j=\jstar+1}^{\jmax}\sqrt{\E_{P} \left(\|\fhat_j-f\|_2^{4}\right)}\sqrt{\mathbb{P}_{P}\left(\jhat=j\right)}.  \label{eqn:wrongchoicelepski_regression}
		\ee

		 Now, by \eqref{eqn:fhat_rosenthal} with $q=2$
		\be 
		\  \E_{P,1}  \left(\fhat_j(\bx)-\E_{P,1} \left(\fhat_j(\bx)\right)\right)^4  
		& \leq C(B_U, B_L,\gamma_{\max}) \Big[\Big(\frac{2^{jd}}{n}\Big)^3+\Big(\frac{2^{jd}}{n}\Big)^2\Big] \\
		&\leq C(B_U, B_L,\gamma_{\max}) \label{eqn:fourthmomentlepski}
		\ee
		by our choice of $2^{\jmax d}$. Also, by standard arguments \cite{hardle1998wavelets}, $|\E_{P,1} \left(\fhat_j(\bx)\right)|=|\Pi\left(f\frac{g}{\ghat}|V_j\right)(\bx)|\leq C(B_U,B_L,\gamma_{\max})$ for all $\bx \in [0,1]^d$.
		Therefore by \eqref{eqn:fourthmomentlepski},
				\be
				\E_P \left(\|\fhat_j-f\|_2^4\right) & \leq 8 \E_{P,2}\Big[\int \E_{P,1} \Big(\fhat_j(\bx)-\E_{P,1} \left(\fhat_j(\bx)\right)\Big)^4 d\bx+\int\left(\E_{P,1} \left(\fhat_j(\bx)\right)-f(\bx) \right)^4d\bx\Big]\\& \leq C(B_U,B_L,\gmax).
				\ee
Also, for any constant $C^{''}$
\be
\ &\mathbb{P}_P\left(\jhat=l\right)  \leq \sum_{j>\jstar}\mathbb{P}_P\Big(\|\fhat_j-\fhat_{\jstar}\|_2> C^{**}\sqrt{\frac{2^{jd}}{n}}\Big)\\
&\leq \sum_{j>\jstar}\E_{P,2}\left\{\begin{array}{c}\mathbb{P}_{P,1}\left(\|\fhat_{\jstar}-\E_{P,1}\left(\fhat_{\jstar}\right)\|_2> \frac{C^{**}}{2}\sqrt{\frac{2^{jd}}{n}}-\|\E_{P,1}\left(\fhat_{\jstar}\right)-\E_{P,1}\left(\fhat_{j}\right)\|_2\right)\\+\mathbb{P}_{P,1}\left(\|\fhat_j-\E_{P,1} \left(\fhat_{j}\right)\|_2> \frac{C^{**}}{2}\sqrt{\frac{2^{jd}}{n}}\right)\end{array}\right\}\\
& \leq \sum_{j>\jstar}\E_{P,2}\left\{\begin{array}{c}\mathbb{P}_{P,1}\left(\|\fhat_{\jstar}-\E_{P,1}\left(\fhat_{\jstar}\right)\|_2> \frac{C^{**}}{2}\sqrt{\frac{2^{jd}}{n}}-\|\Pi\left(f\frac{g}{\hat{g}}|V_{\jstar}\right)-\Pi\left(f\frac{g}{\hat{g}}|V_{j}\right)\|_2\right)\\+\mathbb{P}_{P,1}\left(\|\fhat_j-\E_{P,1} \left(\fhat_{j}\right)\|_2> \frac{C^{**}}{2}\sqrt{\frac{2^{jd}}{n}}\right)\end{array}\right\}\\
& \leq \sum_{j>\jstar}\E_{P,2}\left\{\begin{array}{c}\mathbb{P}_{P,1}\left(\|\fhat_{\jstar}-\E_{P,1}\left(\fhat_{\jstar}\right)\|_2> \left(\frac{C^{**}}{2}\sqrt{\frac{2^{j d}}{n}}-C^{''}\sqrt{\frac{2^{\jstar d}}{n}}\right)\right)\\+\mathbb{P}_{P,1}\left(\|\fhat_j-\E_{P,1} \left(\fhat_{j}\right)\|_2> \frac{C^{**}}{2}\sqrt{\frac{2^{jd}}{n}}\right)\\+\I\left(\|\Pi\left(f\frac{g}{\hat{g}}|V_{\jstar}\right)-\Pi\left(f\frac{g}{\hat{g}}|V_{j}\right)\|_2>C^{''}\sqrt{\frac{2^{j^*d}}{n}}\right)\end{array}\right\}\\
&\leq \sum_{j>j^*}2e^{-C2^{jd/2}}+\sum_{j>\jstar}\P_{P,2}\Big(\|\Pi\Big(f\frac{g}{\hat{g}}|V_{\jstar}\Big)-\Pi\Big(f\frac{g}{\hat{g}}|V_{j}\Big)\|_2>C^{''}\sqrt{\frac{2^{j^*d}}{n}}\Big), \label{eqn:wrongchoiceproblepski_regression}
\ee	 
where the inequality in the last display holds by Lemma \ref{lemma_linear_projection_tail_bound} since $\max\left\{y,\frac{1}{\ghat(\bx)}\right\}\leq C(B_L)$, for a $C>0$  (depending on $M, M', B_U, B_L,\gmax, C^{''}, C^{**}$) if $C^{**}$ is chosen large enough (depending on $M, M', B_U, B_L,\gmax$) such that $C^{**}>2C^{''}$. $C''$ will be chosen later in the proof to be large enough depending on the known parameters of the problem, which in turn will imply that $C^{**}$ can be chosen large enough depending on the known parameters of the problem as well. Finally,
	\be
	\ &  \sum_{j>\jstar}\P_{P,2}\Big(\|\Pi\Big(f\frac{g}{\hat{g}}|V_{\jstar}\Big)-\Pi\Big(f\frac{g}{\hat{g}}|V_{j}\Big)\|_2>C^{''}\sqrt{\frac{2^{j^*d}}{n}}\Big)\\& \leq \sum_{j>\jstar}\P_{P,2}\Big(\|\Pi\left(f|V_{\jstar}\right)-\Pi\left(f|V_{j}\right)\|_2>\frac{C^{''}}{2}\sqrt{\frac{2^{j^*d}}{n}}\Big)\\
	&+\sum_{j>\jstar}\P_{P,2}\Big(\|\Pi\Big(f\Big(\frac{g}{\hat{g}}-1\Big)|V_{\jstar}\Big)-\Pi\Big(f\Big(\frac{g}{\hat{g}}-1\Big)|V_{j}\Big)\|_2>\frac{C^{''}}{2}\sqrt{\frac{2^{j^*d}}{n}}\Big)\\
	&=I+II \label{eqn:wrongchoiceproblepski_regression_finalpiece}
	\ee 
	Since $f \in \besovb(M)$ and choice of $\jstar$, we have from \ref{eqn:holder_approx} that for $C^{''}$ chosen sufficiently large (depending on $M, M'$ and $\gmax$), one has that $I=0$. Control of $II$ is more delicate, but can be handled as below. Using the fact that projection contracts norm, we have
	\be
	II &\leq \sum_{j>\jstar}\P_{P,2}\Big(\|\Pi\Big(f\Big(\frac{g}{\hat{g}}-1\Big)|V_{\jstar}\Big)\|_2>\frac{C^{''}}{4}\sqrt{\frac{2^{j^*d}}{n}}\Big)
	+\sum_{j>\jstar}\P_{P,2}\Big(\|\Pi\Big(f\Big(\frac{g}{\hat{g}}-1\Big)|V_{j}\Big)\|_2>\frac{C^{''}}{4}\sqrt{\frac{2^{j^*d}}{n}}\Big)\\
	& \leq 2 \sum_{j>\jstar}\P_{P,2}\Big(\|\ghat-g\|_2>\frac{B_L' C^{''}}{4B_U}\sqrt{\frac{2^{j^*d}}{n}}\Big). \label{eqref:lepski_f_II_final}
	\ee
The last term in the above display can be bounded using the following lemma.
\begin{lemma}\label{lemma_ghat_prob_bound}
Assume $\gamma_{\min}>\beta_{\max}$. Then for constants $C_1,C_2,C_3>0$ (depending on\\ $M,M',B_U,B_L,\gmax$) one has
\be
\sup_{P \in \Par(\beta,\gamma)}\P_{P,2}\Big(\|\ghat-g\|_2>\frac{B_L'C^{''}}{4B_U}\sqrt{\frac{2^{j^*d}}{n}}\Big)& \leq C_1(\lmax-\lmin)^2\left(e^{-C_2 2^{(\jstar-\lmax)d/2}}+e^{-C_3 2^{\lmin d/2}}\right).
\ee
\end{lemma}

 Plugging in the result of Lemma \ref{lemma_ghat_prob_bound} into \eqref{eqref:lepski_f_II_final}, and thereafter using the facts that $\gmin>\bmax$, $\lmax,\jmax$ are both poly logarithmic in nature,  along with equations \eqref{eqn:wrongchoiceproblepski_regression_finalpiece}, \eqref{eqn:wrongchoiceproblepski_regression}, \eqref{eqn:fourthmomentlepski}, and \eqref{eqn:wrongchoicelepski_regression}, followed by some straightforward but tedious algebra,  we have the existence of an estimator $\ftilde$ depending on $M,M',B_U,B_L,\bmin,\bmax,\gmax$, such that for every $(\beta,\gamma) \in [\betamin,\betamax]\times [\gmin,\gmax]$,
$$\sup\limits_{P \in \mathcal{P}(\beta,\gamma)} \E_P\|\ftilde-f\|_2^2 \leq Cn^{-\frac{2\b}{2\b+d}},$$ 
with a large enough positive constant $C$ depending on  $M,M',B_U,B_L,\bmin, \gmax$.

The proof of $\liminf_{n \to \infty} \inf_{P \in \Par(\beta, \gamma)}\P_P[ \ftilde \in B_{2,\infty}^{\beta}(C)] = 1$, can be done along the lines of the proof of $\liminf_{n \to \infty} \inf_{P \in \Par(\beta, \gamma)}\P_P[ \gtilde \in B_{2,\infty}^{\gamma}(C)] = 1$, since using \eqref{eqn:gtilde_smoothness_prob} and the fact that $\gmin >\beta_{\max}$ one can show using arguments similar to proof of Lemma \ref{lemma:density_truncation} that for sufficiently large $C$, $f \frac{g}{\ghat}\in \besov^{\beta}(C)$,  with suitably high probability uniformly over $P \in \Par(\beta,\gamma)$. 

 The construction of a $\fhat$ from this $\ftilde$ and demonstrating its desired properties is very similar to the derivation of $\ghat$ from $\gtilde$, and hence is omitted.

\noindent Next, we derive the lower bound on the estimation error. The proof will be deferred to the Appendix \ref{sec:regression_lowerbound}. 
\begin{lemma}
\label{lemma:regression_lowerbound}
	There exists a constant $c>0$, independent of $n$, such that 
	\begin{align}
		\inf_{\hat{f}_n} \sup_{P \in \Par(\beta, \gamma)} \E_{P} [ \| \hat{f}_n - f \|_2^2 ] \geq c n^{- \frac{2 \beta}{2\beta + d}}. 
	\end{align}
\end{lemma}

This completes the proof of Theorem \ref{thm:estimation_regression}. 

\section*{Acknowledgments} The authors thank the Associate Editor and two anonymous referees for numerous helpful comments which substantially improved the content and presentation of the paper.

\bibliographystyle{imsart-nameyear}
\bibliography{binary_adaptive_inference_1-1-1}

\appendix
\section{Technical Lemmas} 
Since the estimators arising in this paper also have a linear term, we will need the following standard Bernstein and Rosenthal type tail and moment bounds \citep{petrov1995limit}.
\begin{lemma}\label{lemma_linear_tail_bound}
	If $\bW_1,\ldots,\bW_n \sim \P$ are iid random vectors such that $|L(\bW)|\leq B$ almost surely $\P$, then for $q\geq 2$ one has for large enough constants $C(B)$ and $C(B,q)$
	\be 
	\P(|\frac{1}{n}\sum_{i=1}^n\left(L(\bW_i)-\E(L(\bW_i))\right)|\geq t)\leq 2e^{-nt^2/C(B)},
	\ee
	and 
	\be 
	\ & \E(|\sum_{i=1}^n\left(L(\bW_i)-\E(L(\bW_i))\right)|^q)\\&\leq \left[\sum_{i=1}^n \E\left(|L(\bW_i)-\E(L(\bW_i))|^q\right)+\left[\sum_{i=1}^n \E\left(|L(\bW_i)-\E(L(\bW_i))|^2\right)\right]^{q/2}\right]\leq C(B,q)n^{\frac{q}{2}}.
	\ee
\end{lemma}
We will also need the following concentration inequality for linear estimators based on wavelet projection kernels, proof of which can be done along the lines of proof of Equation (27) of \cite{gine2011rates} or Theorem 5.1.13 of \cite{gine2015mathematical}.
\begin{lemma}\label{lemma_linear_projection_tail_bound}
	Consider i.i.d. observations $\bW_i=(Y,\bX)_i$, $i=1,\ldots,n$ where $\bX_i\in [0,1]^d$ with marginal density $g$. Let $\hat{m}(\bx)=\frac{1}{n}\sum_{i=1}^n L(\bW_i)K_{V_l}\left(\bX_i,\bx\right)$, such that $\max\{\|g\|_{\infty},\|L\|_{\infty}\}\leq B_U$. If $\frac{2^{ld}}{n}\leq 1$, there exists $C,C_1,C_2>0$, depending on $B_U$ and scaling functions $\psi_{0,0}^0,\psi_{0,0}^1$ respectively, such that 
	\be 
	\E(\|\hat{m}-\E(\hat{m})\|_2)\leq C\sqrt{\frac{2^{ld}}{n}},
	\ee
	and for any $x>0$
	\be 
	\P\left(n\|\hat{m}-\E(\hat{m})\|_{2}>\frac{3}{2}n\E(\|\hat{m}-\E(\hat{m})\|_2)+\sqrt{C_1 n 2^{ld/2}x}+C_2 2^{ld/2}x\right)\leq e^{-x}.
	\ee
\end{lemma}

\section{Proofs of U-Statistics Deviation Results} 
{The following tail bound for second order degenerate U-statistics \citep{gine2015mathematical} is due to \cite{gine2000exponential} with constants by \cite{houdre2003exponential} and is crucial for our calculations.}
\begin{lemma}\label{lemma_ustat_tail}
	Let $U_n$ be a degenerate U-statistic of order 2 with kernel $R$ based on an i.i.d. sample $W_1,\ldots,W_n$. Then there exists a constant $C$ independent of $n$, such that 
	\begin{align*}
	P[ | \sum_{i \neq j } R(W_1, W_2) | \geq C( \Lambda_1 \sqrt{u} + \Lambda_2  u + \Lambda_3 u^{3/2} + \Lambda_4 u^2 )] \leq 6 \exp(- u), 
	\end{align*}
	where, we have,
	\begin{align*}
	\Lambda_1^2 &= \frac{n(n-1)}{2} E[ R^2(W_1 , W_2)], \nonumber\\
	\Lambda_2 &= n \sup \{E[R(W_1, W_2) \zeta(W_1) \xi(W_2)] : E[\zeta^2(W_1)]\leq 1, E[\xi^2(W_1)] \leq 1 \}, \nonumber\\
	\Lambda_3 &= \| n E[R^2 (W_1 , \cdot) \| _{\infty} ^{\frac{1}{2}}, \nonumber \\ 
	\Lambda_4 &= \| R \| _{\infty}. \nonumber
	\end{align*}
\end{lemma}
We use this lemma to establish Lemma \ref{lemma_ustat_tail_use}.
\begin{proof}
	Let us analyze $R(\bW_1,\bW_2)=L\left(\bW_1\right) K_{V_j}\left(\bX_1,\bX_2\right)L\left(\bW_2\right)$ first. {{The proof for $R(\bW_1, \bW_2)= L\left(\bW_1\right) K_{W_j}\left(\bX_1,\bX_2\right)L\left(\bW_2\right)$ is analogous}}. By Hoeffding's decomposition one has
	\be
	\ &\frac{1}{n(n-1)}\sum_{i_1 \neq i_2}R\left(\bW_{i_1},\bW_{i_2}\right)-\E\left(R\left(\bW_1,\bW_2\right)\right)
	\\&=\frac{2}{n}\sum_{i_1=1}^n \left[\E_{\bW_{i_1}}R\left(\bW_{i_1},\bW_{i_2}\right)-\E R\left(\bW_{i_1},\bW_{i_2}\right)\right]\\
	&+\frac{1}{n(n-1)}\sum_{i_1 \neq i_2}\left[R\left(\bW_{i_1},\bW_{i_2}\right)-\E_{\bW_{i_1}}R\left(\bW_{i_1},\bW_{i_2}\right)-\E_{\bW_{i_2}}R\left(\bW_{i_1},\bW_{i_2}\right)+\E R\left(
	\bW_{i_1},\bW_{i_2}\right)\right]\\
	&:=T_1+T_2
	\ee
	\subsubsection{Analysis of $T_1$}
	Noting that $T_1=\frac{2}{n}\sum_{i_1=1}^n H(\bW_{i_1})$ where $H(\bW_{i_1})=\E\left(R\left(\bW_{i_1},\bW_{i_2}|\bW_{i_1}\right)\right)-\E R\left(\bW_{i_1},\bW_{i_2}\right)$ we control $T_1$ by standard Hoeffding's Inequality. First note that, 
	\be
	|H(\bW_{i_1})|&=|\sum_{k \in \Z_j}\sum_{v \in \{0,1\}^d}\left[L\left(\bW_{i_1}\right)\psi_{jk}^v\left(\bX_{i_1}\right)\E\left(\psi_{jk}^v\left(\bX_{i_2}\right)L\left(\bW_{i_2}\right)\right)-\left(\E\left(\psi_{jk}^v\left(\bX_{i_2}\right)L\left(\bW_{i_2}\right)\right)\right)^2\right]|\\
	& \leq \sum_{k \in \Z_j}\sum_{v \in \{0,1\}^d}|L\left(\bW_{i_1}\right)\psi_{jk}^v\left(\bX_{i_1}\right)\E\left(\psi_{jk}^v\left(\bX_{i_2}\right)L\left(\bW_{i_2}\right)\right)|+\sum_{k \in \Z_j}\sum_{v \in \{0,1\}^d}\left(\E\left(\psi_{jk}^v\left(\bX_{i_2}\right)L\left(\bW_{i_2}\right)\right)\right)^2
	\ee
	First, by standard compactness argument for the wavelet bases,
	\be
	|\E\left(\psijkv\left(\bX\right)L(\bW)\right)|&\leq \int|\E\left(L(\bW)|\bX=\bx\right)\Big(2^{\frac{jd}{2}}\prod\limits_{l=1}^d\psi_{00}^{v_l}(2^j x_l-k_l )\Big)||g(\bx)|d\bx \\
	&\leq \constant 2^{-\frac{jd}{2}}. \label{eqn:innerproductbound}
	\ee

	Therefore,
	\be
	\sum_{k \in \Z_j}\sum_{v \in \{0,1\}^d}\left(\E\left(\psi_{jk}^v\left(\bX_{i_2}\right)L\left(\bW_{i_2}\right)\right)\right)^2 & \leq \constant \label{eqn:expected_kernel}
	\ee
	Also, using the fact that for each fixed $\bx \in [0,1]^d$, the number indices $k \in \Z_j$ such that $\bx$ belongs to support of at least one of $\psi_{jk}^v$ is bounded by a constant depending only on $\psi^0_{00}$ and $\psi_{00}^1$. Therefore combining \eqref{eqn:innerproductbound} and \eqref{eqn:expected_kernel},
	\be
	\sum_{k \in \Z_j}\sum_{v \in \{0,1\}^d}|L\left(\bW_{i_1}\right)\psi_{jk}^v\left(\bX_{i_1}\right)\E\left(\psi_{jk}^v\left(\bX_{i_2}\right)L\left(\bW_{i_2}\right)\right)|&  \leq \constant 2^{-\frac{jd}{2}}2^{\frac{jd}{2}}=\constant. \\ \label{eqn:linear_term_uniform_bound}
	\ee
	Therefore, by \eqref{eqn:linear_term_uniform_bound} and Hoeffding's Inequality,
	\be
	\P\left(|T_1|\geq t\right)\leq 2e^{-\constant nt^2}. \label{ustat_linear_term_exponential_bound}
	\ee
	\subsubsection{Analysis of $T_2$}
	Since $T_2$ is a degenerate U-statistics, it's analysis is based on Lemma \ref{lemma_ustat_tail}. In particular,
	\be
	T_2&= \frac{1}{n(n-1)}\sum_{i_1 \neq i_2}R^{*}\left(\bW_{i_1},\bW_{i_2}\right)
	\ee
	where 
	\be 
	R^{*}\left(\bW_{i_1},\bW_{i_2}\right)&=\sum_{k \in \Z_j}\sum_{v \in \{0,1\}^d}\left\{\begin{array}{c}\left(L(\bW_{i_1})\psi_{jk}^v\left(\bX_{i_1}\right)-\E\left(\psi_{jk}^v\left(\bX_{i_1}\right)\E\left(L(\bW_{i_1})|\bX_{i_1}\right)\right)\right)\\ \times \left(L(\bW_{i_2})\psi_{jk}^v\left(\bX_{i_2}\right)-\E\left(\psi_{jk}^v\left(\bX_{i_2}\right)\E\left(L(\bW_{i_2})|\bX_{i_2}\right)\right)\right)
	\end{array}\right\}
	\ee
	Letting $\Lambda_i, \ i=1,\ldots,4$ being the relevant quantities as in Lemma \ref{lemma_ustat_tail}, we have the following lemma.
	\begin{lemma}\label{lemma_lamda_estimates}
		There exists a constant $C=\constant$ such that
		$$\Lambda_1^2 \leq C \frac{n(n-1)}{2}2^{jd},\ \Lambda_2 \leq C n,\ \Lambda_3^2 \leq C n 2^{jd},\ \Lambda_4 \leq C 2^{\frac{jd}{2}}.$$
	\end{lemma}	
	\begin{proof}
		First we control $\Lambda_1$. To this end, note that by simple calculations, using bounds on $L,g$, and orthonormality of $\psijkv$'s we have,
		\be
		\  \Lambda_1^2 &=\frac{n(n-1)}{2}\E \left(\left\{R^{*}\left(\bW_1,\bW_2\right)\right\}^2\right)\leq 3n(n-1)\E\left(R^2\left(\bW_1,\bW_2\right)\right)\\
		&=3n(n-1)\E \left(L^2\left(\bW_1\right)K_{V_j}^2\left(\bX_1,\bX_2\right)L^2\left(\bW_2\right)\right)\\&\leq 3n(n-1)B^4 \int \int \Big[\sumk \sumv \psijkv\left(\bx_1\right)\psijkv\left(\bx_2\right)\Big]^2 g(\bx_1)g(\bx_2)d\bx_1 d\bx_2\\
		& \leq 3n(n-1)B^4 B_U^2 \int \int \Big[\sumk \sumv \psijkv\left(\bx_1\right)\psijkv\left(\bx_2\right)\Big]^2 d\bx_1 d\bx_2\\&= 3n(n-1)B^4 B_U^2\sumk \sumv\int  \left(\psijkv\left(\bx_1\right)\right)^2d\bx_2\int  \left(\psijkv\left(\bx_2\right)\right)^2d\bx_2\\
		& \leq \constant n(n-1)2^{jd}.
		\ee
		Next we control
		$\Lambda_2=n\sup\left\{\E \left(R^*\left(\bW_1,\bW_2\right)\zeta\left(\bW_1\right)\xi\left(\bW_2\right)\right): \E \left(\zeta^2(\bW_1)\right)\leq 1, \E \left(\xi^2(\bW_2)\right)\leq 1\right\}.$
		
		To this end, we first control
		\be 
		\ & |\E\left(L(\bW_1)K_{V_j}\left(\bX_1,\bX_2\right)L(\bW_2)\zeta(\bW_1)\xi(\bW_2)\right)|\\
		&=|\int \int \E(L(\bW_1)\zeta(\bW_1)|\bX_1=\bx_1)K_{V_j}\left(\bx_1,\bx_2\right)\E(L(\bW_2)\xi(\bW_2)|\bX_2=\bx_2)g(\bx_2)g(\bx_2)d\bx_1 d\bx_2|\\
		&=|\int \E(L(\bW)\zeta(\bW)|\bX=\bx) \Pi\left(\E(L(\bW)\xi(\bW)|\bX=\bx)g(\bx)|V_j\right)g(\bx)d\bx|\\
		& \leq \left(\int \E^2(L(\bW)\zeta(\bW)|\bX=\bx)g^2(\bx)d\bx\right)^{\frac{1}{2}}\left(\int \Pi^2\left(\E(L(\bW)\xi(\bW)|\bX=\bx)g(\bx)|V_j\right)d\bx\right)^{\frac{1}{2}} \\
		& \leq \left(\int \E(L^2(\bW)\zeta^2(\bW)|\bX=\bx)g^2(\bx)d\bx\right)^{\frac{1}{2}}\left(\int \E(L^2(\bW)\xi^2(\bW)|\bX=\bx)g^2(\bx)d\bx \right)^{\frac{1}{2}}\\
		& \leq B^2 B_U \sqrt{\E(\zeta^2(\bW_1))\E (\xi^2(\bW_2))}\leq B^2 B_U
		\ee
		Above we have used Cauchy-Schwartz Inequality, Jensen's Inequality, and the fact that projections contract norm. Also,
		\be
		\ & |\E \left(\E \left(L(\bW_1)K_{V_j}\left(\bX_1,\bX_2\right)L(\bW_2)|\bW_1\right)\zeta(\bW_1)\xi(\bW_2)\right)|\\
		& = |\E \left[L(\bW_1)\Pi\left(\E \left(L(\bW_1)g(\bX_1)|\bX_1\right)|V_j\right)\zeta(\bW_1)\xi(\bW_2)\right]|\\
		& = |\E \left[L(\bW_1)\Pi\left(\E \left(L(\bW_1)g(\bX_1)|\bX_1\right)|V_j\right)\zeta(\bW_1)\right]||\E(\xi(\bW_2))|\\
		& \leq |\int \Pi(\E (L(\bW)\zeta(\bW)|\bX=\bx)g(\bx)|V_j)\Pi(\E (L(\bW)|\bX=\bx)g(\bx)|V_j)d\bx| \leq B^2 B_U,
		\ee
		where the last step once again uses contraction property of projection, Jensen's Inequality, and bounds on $L$ and $g$. Finally, by Cauchy-Schwartz Inequality and \eqref{eqn:expected_kernel},
		\be
		\E \left[\E \left(L(\bW_1)K_{V_j}\left(\bX_1,\bX_2\right)L(\bW_2)\right)\zeta(\bW_1)\xi(\bW_2)\right]& \leq \sumk \sumv \E^2 \left(L(\bW)\psijkv(\bX)\right)\leq \constant.
		\ee
		This completes the proof of $\Lambda_2 \leq \constant n$.
		Turning to $\Lambda_3=n \|\E\left[\left(R^*(\bW_1,\cdot)\right)^2\right]\|^{\frac{1}{2}}_{\infty}$ we have that
		\be 
		\left(R^*(\bW_1,\bw_2)\right)^2 & \leq 2 \left[R(\bW_1,\bw_2)-\E (R(\bW_1,\bW_2)|\bW_1)\right]^2+2\left[\E (R(\bW_1,\bW_2)|\bW_2=\bw_2)-\E \left(R(\bW_1,\bW_2)\right)\right]^2
		\ee
		Now,
		\be 
		\ & \E \left[R(\bW_1,\bw_2)-\E (R(\bW_1,\bW_2)|\bW_1)\right]^2 \\& \leq 2 \E \left(L^2(\bW_1)K^2_{V_j}\left(\bX_1,\bx_2\right)L^2(\bw_2)\right)
		+2\E \Big(\sumk\sumv L(\bW_1)\psijkv(\bX_1)\E \left(\psijkv(\bX_2)L(\bW_2)\right)\Big)^2\\
		& \leq 2B^4 B_U^2 \sumk\sumv \left(\psijkv(\bx_2)\right)^2+2\E (H^2(\bW_2))\leq \constant 2^{jd}.
		\ee
		where the last inequality follows from arguments along the line of \eqref{eqn:linear_term_uniform_bound}. Also, using inequalities \eqref{eqn:expected_kernel} and \eqref{eqn:linear_term_uniform_bound}
		\be
		\ & \left[\E (R(\bW_1,\bW_2)|\bW_2=\bw_2)-\E \left(R(\bW_1,\bW_2)\right)\right]^2\\&=\Big[\sumk \sumv \E \left(L(\bW_1)\psijkv(\bX_1)\right)\left(\E\left(L(\bW_1)\psijkv(\bX_1)\right)-\psijkv(\bx_2)L(\bw_2)\right)\Big]^2\leq \constant.
		\ee
		This completes the proof of controlling $\Lambda_3$. Finally, using compactness of the wavelet basis,
		\be
		\|R(\cdot,\cdot)\|_{\infty}& \leq B^2 \sup_{\bx_1,\bx_2}\sumk \sumv |\psijkv(\bx_1)||\psijkv(\bx_2)|\leq \constant 2^{jd}
		\ee
		Combining this with arguments similar to those leading to \eqref{eqn:linear_term_uniform_bound}, we have $\Lambda_4\leq \constant 2^{jd}$.
	\end{proof}
	Therefore, using Lemma \ref{lemma_ustat_tail} and Lemma \ref{lemma_lamda_estimates} we have 
	\be 
	\P\Big(|T_2|\geq \frac{\constant}{n-1}\Big(\sqrt{2^{jd}t}+t+\sqrt{\frac{2^{jd}}{n}}t^{\frac{3}{2}}+\frac{2^{jd}}{n}t^2\Big)\Big)& \leq 6 e^{-t}.
	\ee
	Finally using $2t^{\frac{3}{2}}\leq t+t^2$ we have,
	\be
	\Pf\left[|T_2|> a_1\sqrt{t}+a_2t+a_3 t^2\right]& \leq 6 e^{-t} \label{ustat_quadratic_term_bound}
	\ee
	where $a_1= \frac{\constant}{n-1}2^{\frac{jd}{2}}$, $a_2=\frac{\constant}{n-1}\left(\sqrt{\frac{2^{jd}}{n}}+1\right)$, $a_3=\frac{\constant}{n-1}\left(\sqrt{\frac{2^{jd}}{n}}+\frac{2^{jd}}{n}\right)$. Now if $h(t)$ is such that $a_1\sqrt{h(t)}+a_2 h(t)+a_3h^2(t)\leq t$, then one has by \eqref{ustat_quadratic_term_bound},
	\be
	\mathbb{P}\left[|T_2|\geq t\right]& \leq \mathbb{P}\left[|T_2|\geq a_1\sqrt{h(t)}+a_2 h(t)+a_3h^2(t)\right]\leq 6 e^{-6h(t)}.
	\ee
	Indeed, there exists such an $h(t)$ such that $h(t)=b_1 t^2 \wedge b_2 t \wedge b_3 \sqrt{t}$ where $b_1=\frac{\constant}{a_1^2}$, $b_2=\frac{\constant}{a_2}$, and $b_3=\frac{\constant}{\sqrt{a_3}}$. Therefore, there exists $C=\constant$ such that 
	\be 
	\mathbb{P}\left[|T_2|\geq t\right]\leq e^{-\frac{C t^2}{a_1^2}}+e^{-\frac{C t}{a_2}}+e^{-\frac{C \sqrt{t}}{\sqrt{a_3}}}. \label{ustat_quadratic_term_bound_use}
	\ee
	
	\subsubsection{Combining Bounds on $T_1$ and $T_2$} Applying union bound along with \ref{ustat_linear_term_exponential_bound} and \ref{ustat_quadratic_term_bound_use} completes the proof of Lemma \ref{lemma_ustat_tail_use}.
\end{proof}

\section{Remaining Technical Details for Adaptive Estimation}\label{section_appendix_adaptive_estimation}
	
	
\subsection{Proof of Lemma \ref{lemma_adaptive_gtilde}}
\label{sec:proof_lemma_adaptive_gtilde}
To analyze the estimator $\gtilde$, we begin with standard bias variance analysis for the candidate estimators $\ghat_l$.

Note that for any $\bx \in [0,1]^d$, using standard facts about compactly supported wavelet basis having regularity larger than $\gamma_{\max}$ \citep{hardle1998wavelets}, one has for a constant $C_1$ depending only on the wavelet basis used,
\be
\|\E_P\left(\ghat_l\right)-g\|_2^2&=\|\Pi\left(g|V_l\right)-g\|_2^2\leq C_1^2 M'^2 2^{-2ld\frac{\g}{d}}. \label{eqn:qbias_lepski_density}
\ee
Above we have used the fact that 
\be
\sup\limits_{h \in \besov^{\g}(M)}\|h-\Pi(h|V_l)\|_{2}\leq C_1 M' 2^{-l\gamma}. \label{eqn:holder_approx}
\ee
Also by Rosenthal's Inequality \citep{petrov1995limit}, there exists a constant $C(q)$ for $q\geq 2$ such that
\be
\ & \E_P\left(|\ghat_l(\bx)-\E_P\left(\ghat_l(\bx)\right)|^q\right)\\
&\leq \frac{C(q)}{n^q}\Big[\sum_{i=n+1}^{2n}\E_P\left(|K_{V_l}(\bx_i,\bx)|^q\right)+\Big(\sum_{i=n+1}^{2n}\E_P\Big(|K_{V_l}(\bx_i,\bx)|^2\Big)\Big)^{q/2}\Big]\\
& \leq \frac{C_2^q/2}{n^q}\times \Big[n\left(2^{ld}\right)^{q-1}+n^{q/2}\left(2^{ld}\right)^{q/2}\Big], 
\ee
where the last inequality follows using standard facts about compactly supported wavelet basis having regularity larger than $\gamma_{\max}$ \citep{hardle1998wavelets} with a constant $C_2$ that depends only on $q$ and the wavelet basis used. Therefore, for $q \geq 2$, by the choice of $l \in \mathcal{T}_2$, we have that for all $\bx \in [0,1]^d$, 
\be
\E_P\left(|\ghat_l(\bx)-\E_P\left(\ghat_l(\bx)\right)|^q\right)&\leq C_2^q\Big(\frac{2^{ld}}{n}\Big)^{q/2}.\label{eqn:qvariance_lepski_density}
\ee
Therefore, we have the following bias-variance  decomposition.
\be 
\ & \E_P\left(\|\ghat_l-g\||_2^2\right)=\int\E_P\left(|\ghat_l(\bx)-g(\bx)|^2\right)d\bx\\
&= \Big[\int\E_P\left(|\ghat_l(\bx)-\E_P\left(\ghat_l(\bx)\right)|^2\right)d\bx+\int\E_P\left(|\E_P\left(\ghat_l(\bx)\right)-g(\bx)|^2\right)d\bx\Big]\\
& \leq C_1^2 M'^2 2^{-2l\gamma}+C_2^2\left(\frac{2^{ld}}{n}\right)\label{eqn:qthmomentlepski_g}
\ee

Let $\lstar=\min\left\{l:C_1 M' 2^{-l\gamma}\leq C_2\sqrt{\frac{2^{ld}}{n}} \right\}$. This implies that
\be
\left\|\E_P\left(\ghat_{\lstar}\right)-g\right\|_2^2\leq C_1^2 M'^2 2^{-2\lstar  \gamma}\leq C_2^2\Big(\sqrt{\frac{2^{\lstar d}}{n}}\Big)^2\leq 2^dC_2^2  \Big(\frac{C_1}{C_2}M'\Big)^{\frac{2d}{2\gamma+d}}n^{-\frac{2\gamma}{2\gamma+d}}
\ee

Therefore, by definition of $\lhat$ and $\lstar$,
\be
\ &\E_P\left(\|\gtilde-g\|_2^2\I\left(\lhat \leq \lstar\right)\right)
=\E_{P,2}\left(\|\gtilde-g\|_2^2\I\left(\lhat \leq \lstar\right)\right)\\
&\leq  2\E_{P,2}\left(\|\gtilde-\ghat_{\lstar}\|_2^2\I\left(\lhat \leq \lstar\right)\right)+2 \E_{P,2}\left(\|\ghat_{\lstar}-g\|_2^2\right)\\
&\leq 2^{d+1}\left((C^*)^2+2\right)C_2^2  \Big(\frac{C_1}{C_2}M'\Big)^{\frac{2d}{2\gamma+d}}n^{-\frac{2\gamma}{2\gamma+d}}. \label{eqn:lessthanlstar}
\ee

Using Cauchy-Schwarz inequality, we have, 
\be
\E_P\left(\|\gtilde-g\|_2^2\I\left(\lhat > \lstar\right)\right)& \leq \sum\limits_{l=\lstar}^{\jmax}\sqrt{\E_{P,2} \left(\|\ghat_l-g\|_2^{4}\right)}\sqrt{\mathbb{P}_{P,2}\left(\lhat=l\right)}. \label{eqn:wrongchoicelepski_g}
\ee
Now, by \eqref{eqn:qbias_lepski_density}, \eqref{eqn:qvariance_lepski_density}, choice of $l \in \mathcal{T}_2$, and Jensen's Inequality
\be 
\E_{P,2} \left(\|\ghat_l-g\|_2^{4}\right)& = \E_{P,2}\left(\int |\ghat_l(\bx)-g(\bx)|^2 d\bx\right)^2
 \leq \E_{P,2}\int |\ghat_l(\bx)-g(\bx)|^{4} d\bx\\
& \leq C_1^{4}M'^{4}2^{-4l\gamma}+C_2^{4}\left(\frac{2^{ld}}{n}\right)^{2}
 \leq C_1^{4}M'^{4}+C_2^{4}.
\ee

Next, note that for $l>\lstar$,
\be
\ & \mathbb{P}_{P,2}\left(\lhat=l\right)  \leq \sum_{l>\lstar}\mathbb{P}_{P,2}\Big(\|\ghat_l-\ghat_{\lstar}\|_2> C^*\sqrt{\frac{2^{ld}}{n}}\Big)\\
&\leq \sum_{l>\lstar}\left\{\begin{array}{c}\mathbb{P}_{P,2}\left(\|\ghat_{\lstar}-\E_{P,2}\left(\ghat_{\lstar}\right)\|_2> \frac{C^*}{2}\sqrt{\frac{2^{ld}}{n}}-\|\E_{P,2}\left(\ghat_{\lstar}\right)-\E_{P,2}\left(\ghat_{l}\right)\|_2\right)\\+\mathbb{P}_{P,2}\left(\|\ghat_l-\E_{P,2} \left(\ghat_{l}\right)\|_2> \frac{C^*}{2}\sqrt{\frac{2^{ld}}{n}}\right)\end{array}\right\}\\
& \leq \sum_{l>\lstar}\left\{ \begin{array}{c}\mathbb{P}_{P,2}\left(\|\ghat_{\lstar}-\E_{P,2}\left(\ghat_{\lstar}\right)\|_2> \frac{C^*}{2}\sqrt{\frac{2^{ld}}{n}}-\|\Pi\left(g|V_{\lstar}\right)-\Pi\left(g|V_{l}\right)\|_2\right)\\+\mathbb{P}\left(\|\ghat_l-\E_{P,2} \left(\ghat_{l}\right)\|_2> \frac{C^*}{2}\sqrt{\frac{2^{ld}}{n}}\right)\end{array}\right\}\\
& \leq \sum_{l>\lstar}\left\{ \begin{array}{c}\mathbb{P}_{P,2}\left(\|\ghat_{\lstar}-\E_{P,2}\left(\ghat_{\lstar}\right)\|_2> \frac{C^*}{2}\sqrt{\frac{2^{ld}}{n}}-2C_2\sqrt{\frac{2^{l^*d}}{n}}\right)\\+\mathbb{P}\left(\|\ghat_l-\E_{P,2} \left(\ghat_{l}\right)\|_2> \frac{C^*}{2}\sqrt{\frac{2^{ld}}{n}}\right)\end{array}\right\}\\
& \leq \sum_{l>\lstar}\left\{ \begin{array}{c}\mathbb{P}_{P,2}\left(\|\ghat_{\lstar}-\E_{P,2}\left(\ghat_{\lstar}\right)\|_2> (\frac{C^*}{2}-2C_2)\sqrt{\frac{2^{ld}}{n}}\right)\\+\mathbb{P}\left(\|\ghat_l-\E_{P,2} \left(\ghat_{l}\right)\|_2> \frac{C^*}{2}\sqrt{\frac{2^{ld}}{n}}\right)\end{array}\right\}
\leq \sum_{l>l^*}2e^{-C2^{ld/2}}, \label{eqn:wrongchoiceproblepski_g}
\ee
for a $C>0$ (depending on $B_U$ and the wavelet basis choice) if $C^*$ is chosen large enough (depending on $M'$ and $B_U$) such that $C^*>2C_2$. In the fourth and fifth of the above series of inequalities, we have used \eqref{eqn:holder_approx} and the definition of $\lstar$ respectively.
The last line follows by  an argument similar to results in Section 3.1 of \cite{gine2011rates}. Finally combining equations \eqref{eqn:lessthanlstar}, \eqref{eqn:wrongchoicelepski_g} and \eqref{eqn:wrongchoiceproblepski_g}, we have the existence of an estimator $\gtilde$ depending on $B_U$, $\gmin$, and $\gmax$, such that for every $(\beta,\gamma) \in [\betamin,\betamax]\times [\gmin,\gmax]$,
$$\sup\limits_{P \in \mathcal{P}(\beta,\gamma)} \E_P\|\gtilde-g\|_2^2 \leq  Cn^{-\frac{2\gamma}{2\gamma+d}},$$ 
with a large enough positive constant $C$ depending on $M,B_U,\gmin$. 
\subsection{Proof of Lemma \ref{lemma:density_truncation}}
\label{proof_density_truncation}
We will utilize the equivalent definition of Besov space in terms of moduli of smoothness. We define the forward difference operator $\Delta_h(f)(x) = f(x+h) - f(x)$ and the operator $\Delta_h^r = \Delta_h ( \Delta_h^{r-1})$ for $r\geq 2$, where $\Delta_h^1 = \Delta$. Next, for $t >0$ and $r$ a natural number greater than $\beta$, we define the modulus of smoothness 
$\omega_r(f, t) = \sup_{|h| \leq t} \| \Delta_h^r(f) \|_2$. 
Finally, we define the Besov semi-norm $|f|_{\besovb}= \sup_{t >0} \omega_r(f,t)/ t^\beta$.  Finally, we define
\be
\besovb(M) = \{ f \in L^2 : \|f \|_{\besovb}= \|f \|_2 + |f|_{\besovb} \leq M \}. \label{eq:besov_equiv}
\ee
It is a standard fact \citep{hardle1998wavelets} that \eqref{eq:besov_equiv} is an equivalent definition of a Besov space. Further, the supremum in the definition of $|f|_{\besovb}$ may be restricted to $0<t<1$. Throughout this proof, we work with $\besovb(M)$ defined by \eqref{eq:besov_equiv} without loss of generality. 
We first consider the case when $0 < \beta <1$. In this case, it is easy to see that $\| \phi(f) \|_2 < C(\phi)$, for some universal constant $C(\phi)$ depending on $\phi$ and independent of $f$. Next, we control the term $|\phi(f) |_{\besovb}$.  Using Mean Value Theorem, we have, 
\begin{align}
\Delta_h(\phi(f))(x) = \phi(f(x+h)) - \phi(f(x)) = \phi'(\xi) \Delta_h(f)(x), \nonumber 
\end{align}
for some $\xi \in [\min\{f(x), f(x+h)\}, \max\{f(x),f(x+h)\}]$. This naturally implies 
$\omega_1(\phi(f), t) \leq \|\phi \|_{\infty} \omega_1(f,t)$, which gives us the desired claim in this case.

Next, we consider the case when $\beta >1$. We note that for any $r\geq 1$, we have,
$\Delta_h^r(f)(x) = \sum_{k=0}^{r} {r \choose k} (-1)^{r-k} f(x+ kh)$. Setting $r= \lceil \beta \rceil$, we have, by Taylor expansion for $\phi$, 
\begin{align}
\Delta_h^r(\phi(f))(x) &= \sum_{k=0}^{r} {r \choose k } (-1)^{r-k} \phi(f(x+kh))  \nonumber \\
&= \sum_{k=0}^{r} {r \choose k} (-1)^{r-k} \Big[\phi'(f(x)) \Delta_{kh}(f)(x) + \frac{\phi''(\xi(x))}{2} (\Delta_{kh}(f)(x))^2\Big] \nonumber \\
&= \phi'(f(x)) \Delta_h^r(\phi(f))(x) + \Sigma(x,h). \nonumber 
\end{align}
Thus we have, $\| \Delta_h^r(\phi(f)) \|_2 \leq \|\phi' \|_{\infty} \|\Delta_h^r(\phi(f))\|_2 + \| \Sigma(\cdot, h)\|_2$. To control $\|\Sigma\|_2$, we use the fact that $\besovb (M) \subset B_{\infty,\infty}^{\beta- 1/2}(C(M,\beta))$, where $\subset$ stands for the usual embedding operation (results of similar flavor can be found in \cite{simon1990sobolev, triebel2006theory} ). This naturally implies that 
$ (\Delta_{kh}(f)(x))^2 \lesssim (kh)^{2\beta-1}$. Thus,we have, 
\begin{align}
\sup_{0<t<1} \frac{\sup_{|h| \leq t} \| \Sigma(\cdot, h)\|_2}{t^\beta} \leq C(M, \beta) t^{\beta -1}. \nonumber 
\end{align}
This completes the proof. 

\subsection{Proof of Lemma \ref{lemma_ghat_prob_bound}}
Indeed, $\ghat=\psi(\gtilde)$, where $\psi(x)$ is $C^{\infty}$ function which is identically equal to $x$ on $[B_L,B_U]$ and has universally  bounded first derivative. Therefore, it is enough to prove Lemma \ref{lemma_ghat_prob_bound} for $\gtilde$ instead of $\ghat$ and thereby invoking a simple first order Taylor series argument along with the fact that $\psi(g)\equiv g$ owing to the bounds on $g$. The proof of the lemma is therefore very similar to the proof of adaptivity of $\ghat$ (by dividing into cases where the chosen $\lhat$ is larger and smaller than $\lstar $ respectively and thereafter invoking Lemma \ref{lemma_linear_projection_tail_bound}) and therefore we simply state the main idea and omit the details.  The crux of the argument for proving Lemma \ref{lemma_ghat_prob_bound}  relies on the fact that by Lemma \ref{lemma_linear_projection_tail_bound}, any $\ghat_l$ for $l \in \mathcal{T}_2$ suitably concentrates around $g$ in a radius of the order of $\sqrt{\frac{2^{ld}}{n}}$, and Lepski's method chooses an index $\lhat \leq \lstar$ with high probability. Thereafter one uses the fact that $\gmin >\bmax$, and consequently $2^{ld}\ll 2^{jd}$ for any $(j,l)\in \mathcal{T}_1\times\mathcal{T}_2 $. 

\subsection{Proof of Lemma \ref{lemma:regression_lowerbound}}
\label{sec:regression_lowerbound}
	The proof will follow the usual approach of lower bounding the estimation error by a related ``testing" problem \citep{tsybakov2008introduction}. We will equip our parameter space with the distance function $d((f,g) , (f' , g')) = \sqrt{ \| f - f' \|_2^2 + \| g - g' \|_2^2 }$.
	
	We will use $M$ distributions in our derivation of the lower bound --- $M$ will be chosen appropriately later. 
	The distributions $\mathcal{C}=\{ (f_i , g_i) : 1\leq i \leq M \}$ are chosen as follows: we set $g_i =1$ for all $i$, that is,
	we set the design density to be uniform. Next, we set $j_0 = \lceil  \frac{ d}{2 \beta + d} \log_2 n \rceil$. Let 
	\begin{align}
		f_i(x) = \frac{1}{2} + \varepsilon 2^{- j_0 ( 1/2+\beta/d) } \sum_{ k \in \mathcal{Z}_{j_0}} \sum_{ v \in \{ 0, 1\}^d - \{0\} } \alpha_{i, k}^v \psi_{j_0, k}^v (x), \nonumber  
	\end{align}
	where each $\alpha_{i,k}^v \in \{0,1\}$. The constant $\varepsilon >0$ is chosen sufficiently small such that $0\leq f_i \leq 1$ for all $x \in [0,1]^d$.  Thus we have, for $(f,g) , (f', g') \in \mathcal{C}$, 
	\begin{align*}
		d((f,g), (f', g'))^2 = \| f - f' \|_2^2  = \varepsilon^2 \frac{1}{n} \sum_{ v \in \{0,1 \}^d - \{0\} } \rho ( \alpha_{i \cdot}^v , \alpha_{i' \cdot}^v ), 
	\end{align*}
	where $\alpha_{i \cdot}^v = (\alpha_{i,k}^v)$ and $\rho(\cdot, \cdot)$ is the Hamming distance between two vectors on the hypercube. For each $v \in \{0,1\}^d - \{0\}$, we apply the Varshamov-Gilbert Lemma ({Lemma 2.9 of \cite{tsybakov2008introduction}}) to select $(\alpha_{i, \cdot}^v)$ with mutual separation at least  $\frac{1}{8} n^{\frac{d}{2\beta + d}}$. The Varshamov-Gilbert Lemma guarantees the existence of such a subset with size at least $2^{\frac{1}{8}n^{\frac{d}{2 \beta +d }}}$. Thus we have, with $M= 2^{\frac{2^d -2}{8}n^{\frac{d}{2 \beta +d }}}$, for any $(f,g), (f',g') \in \mathcal{C}$, 
	\begin{align*}
		d((f,g), (f', g'))^2 \geq \frac{(2^d-2)\varepsilon^2}{8} n^{-\frac{2\beta}{2\beta + d}}. 
	\end{align*}
	We denote the joint distribution of $\{ \mathbf{x}_l, y_l : 1 \leq l \leq n\}$ under the parameters $(f_i, g_i)$ by $\mathcal{P}_i$.  Thus we have, $\chi^2 (\mathcal{P}_i, \mathcal{P}_0) = [1+ \chi^2((f_i, g_i), (f_0, g_0))]^n - 1$.
	
	Finally, we note that 
		$1+ \chi^2((f_i,g_i), (f_0, g_0)) = \E_0\Big[ \Big( \frac{f_i(\mathbf{x}_1)^{y_1} (1- f_i(\mathbf{x}_1))^{(1-y_1)}}{1/2}\Big)^2 \Big] 
		= 4 \E_0[ f_i^2(\mathbf{x}_1) + (1- f(\mathbf{x}_i))^2 ]$, 
	where $\E_0[\cdot]$ represents the expectation with respect to $(f_0, g_0)$. Setting $f_i = 1/2 + \psi_i$, we have,
	\begin{align*}
		1+ \chi^2((f_i,g_i), (f_0, g_0))  = 1 + 4 \E_0[\psi_i(\mathbf{x}_1)^2] \leq 1 + 4 (2^d-2)\varepsilon^2 n^{-\frac{2\beta}{2\beta +d }}. 
	\end{align*}
	Thus
$\chi^2 (\mathcal{P}_i, \mathcal{P}_0) \leq \exp(4 (2^d-2) \varepsilon^2 n^{\frac{d}{2\beta + d}}) \leq \delta M$,
	for some $0<\delta < 1/8$ if $\varepsilon>0$ is chosen sufficiently small. 
	This allows us to complete the proof by an application of Theorem 2.7 in \cite{tsybakov2008introduction}.

\end{document}